\documentclass[letterpaper,12pt,reqno]{amsart}
\usepackage{amssymb}
\usepackage{cases}
\usepackage{amsmath}
\usepackage{mathrsfs}
\usepackage{amsfonts}
\usepackage{amsfonts}
\usepackage[all]{xy}
\usepackage{color}
\DeclareFontFamily{OT1}{pzc}{}
\DeclareFontShape{OT1}{pzc}{m}{it}{<-> [1.15] rpzcmi}{}
\DeclareMathAlphabet{\mathzc}{OT1}{pzc}{m}{it}

\usepackage{color}

\def\End{\operatorname{End}\kern-.5pt}

\def\KK{{\mathzc K\kern0pt}}
\def\vv{{\mathzc v\kern.5pt}}

\def\aq{/\kern-2pt/}
\def\Hom{\operatorname{Hom}}

\def\La{{\Lambda}}

\def\fS{{\mathfrak S}}
\def\End{{\text{\rm End}}}
\def\Hom{{\text{\rm Hom}}}

\def\sZ{{\mathcal Z}}

\def\sI{{\mathcal I}}
\def\sJ{{\mathcal J}}

\def\sR{{\mathcal R}}

\def\sB{{\mathcal B}}

\def\sQ{{\mathcal Q}}

\newtheorem{theorem}{Theorem}[section]
\newtheorem{lemma}[theorem]{Lemma}
\newtheorem{proposition}[theorem]{Proposition}
\newtheorem{corollary}[theorem]{Corollary}
\theoremstyle{definition}
\newtheorem{definition}[theorem]{Definition}

\newtheorem{remark}[theorem]{Remark}

\numberwithin{equation}{theorem}

\def\bZ{{\mathbb Z}}

\def\fp{{\mathfrak p}}

\def\scc{{\textsc{c}}}

\title[Presenting Hecke endomorphism algebras]{Presenting Hecke endomorphism algebras by Hasse quivers with relations}

\date{12 Dec 2014}

\author{Jie Du}
\address{JD: School of Mathematics and Statistics,
University of New South Wales, Sydney NSW 2052, Australia}
\email{j.du@unsw.edu.au}

\author{Bernt Tore Jensen}
\address{BTJ: NTNU, Norwegian University of Science and Technology, Teknologvn. 22, 2815 Gj\o vik, Norway}
\email{bernt.jensen@ntnu.no}

\author{Xiuping Su}
\address{XS: Mathematical Sciences, University of Bath, Bath BA2 7AY, U.K.}
\email{xs214@bath.ac.uk}

\thanks{The authors gratefully acknowledge support from ARC and EPSRC under grants DP120101436 and EP/1022317/1. The work was initiated during the first author's visit to the University of Bath at the end of 2013 and was completed while the second and third authors were visiting UNSW on their sabbatical leave in 2015. We
would like to thank UBath and UNSW for their hospitality during the writing of the paper.}

\date{\today}

\begin{document}

\baselineskip16pt
\def\hei{\relax}
\maketitle

\begin{abstract} A Hecke endomorphism algebra is a natural generalisation of the $q$-Schur algebra associated with the symmetric group to a Coxeter group. For Weyl groups, B. Parshall, L. Scott and the first author \cite{DPS,DPS4} investigated the stratification structure of these algebras in order to seek applications to representations of finite groups of Lie type. 
In this paper we investigate the presentation problem for Hecke endomorphism algebras associated with arbitrary Coxeter groups.
Our approach is to present such algebras by quivers with relations. If $R$ is the localisation of $\mathbb Z[q]$ at the polynomials with the constant term 1, the algebra can simply be defined by the so-called idempotent, sandwich and extended braid relations. As applications of this result,
we first obtain  a presentation of the 0-Hecke endomorphism algebra over $\mathbb{Z}$ and then develop an algorithm for presenting the Hecke endomorphism algebras
over $\mathbb Z[q]$ by finding torsion relations. As examples, we determine the torsion relations required for all rank 2 groups and the symmetric group $\fS_4$.
\end{abstract}

\section{Introduction}
Since I. Schur first introduced Schur algebras for linking representations of the general linear group $GL_n(\mathbb C)$, a continuous group,
with those of the symmetric group $\mathfrak S_r$, a finite group, this class of algebras has been generalised to various objects in several directions. For example, by the role they play in the Schur--Weyl duality, there are affine and super generalisations  and their quantum analogues, which are known as (affine, super) $q$-Schur algebras (see, e.g.,
\cite{Ji}, \cite{DJ}, \cite{GV}, \cite{Mi}). As homomorphic images of a universal enveloping algebra, generalised Schur algebras and their quantum analogues are investigated not only for type $A$ but also for other types (see, e.g., \cite{Do}, \cite{DS}). By their definition as endomorphism algebras of certain permutation modules of symmetric groups, they have a natural generalisation to Hecke endomorphism algebras associated with finite Coxeter groups (see, e.g., \cite{Du1}, \cite{DPS}, \cite{DS2}).

Recently, with a completely different motivation, advances have been made in categorifying Hecke endomorphism algebras. For example, M. Mackaay, et al, obtained a diagrammatic categorification of the $q$-Schur algebra \cite{MSV} and some affine counterpart \cite{MT}, while G. Williamson \cite{GW} investigated a more general class of Hecke categories associated with any Coxeter systems $(W,S)$ (called Schur algebroids loc.cit.), and categorified them in terms of  singular Soergel bimodules. As seen from these works, a presentation of an algebra by generators and relations
is related to the categorification of Hecke-endomorphism algebras.
In fact, the question of presenting the Hecke endomorphism algebras by generators and relations was 
raised in Remark 2.2 in \cite{GW}.

Unlike the $q$-Schur algebras, Hecke endomorphism algebras do not have a direct connection to Lie algebras and quantum groups and thus cannot be presented as a quotient of a quantum group. However, by viewing them as enlarged Hecke algebras, they can be presented with generators and relations, which are rooted in Hecke (algebra) relations. This paper is going to tackle the presentation problem in this direction.

Our approach is to present these algebras by quivers with relations, following closely the idea of using Hecke relations. First, by assuming the invertibility of all Poincar\'e polynomials in the ground ring $R$, we consider in Sections 3 and 4 a quiver $Q'$ with loops and impose the Hecke relations directly on the loops. We then replace the loops by cycles of length 2 to obtain a quiver $Q$ with no loops and braid relations are imposed explicitly again. Now, by thinking of the generators displayed in \cite[Cor.~2.12]{GW}, we take a subset labelled by $I,J\subseteq S$ satisfying $I\sqsubset J$ (see {\eqref{sqsubset}) and introduce in Section 5 a Hasse quiver $\tilde Q$ to replace $Q$. It turns out that the algebra over $R$ can simply be presented by $\tilde Q$ together with the so-called extended braid relations and some obvious relations,  called idempotent relations and sandwich relations (see Theorem \ref{tildeQ}).

The first application of this result is to obtain a presentation of the 0-Hecke endomorphism algebras by specialising $q$ to 0 in Section 6. These degenerate algebras have attracted some attention (see, e.g., \cite{Norton}, \cite{carter}, \cite{fayers}, \cite{Reineke},\cite{JS}, \cite{JSY}, \cite{HNT}, \cite{DY}) and
 have also nice applications. For instance,  J. Stembridge  \cite{ST} used the 0-Hecke algebra to
give a derivation of the M$\ddot{o}$bius function of the Bruhat order, while X. He  \cite{He} gave a more elementary construction of 
a monoid structure by A.  Berenstein and D. Kazhdan \cite{BK}.  

We then investigate the integral case in Sections 7 and 8. We  analyse the gap between the presentation over $R$ and a possible presentation over the polynomial ring $\mathbb Z[q]$. The idempotent relations can be replaced by the so-called quasi-idempotent relations, the sandwich relations are unchanged. The challenge is how to replace the extended braid relations by some torsion relations. We develop an algorithm and compute the examples of rank 2 and of type $A_3$. The rank 2 case is relatively easy, the required torsion relations are simply the refined braid relations. Note that a recursive version of this case is done by B. Elias \cite[Prop. 2.20]{Elias}. However, the $A_3$ case is more complicated. On top of the refined braid relations, there are two more sets of torsion relations, see Theorem \ref{A3}. We believe that the algorithm can be used to compute the other lower rank cases.


\section{Hecke  algebras and Kazhdan--Lusztig generators}\label{section2}
 Let $(W,S)$ be a Coxeter system and let
$\ell:W\to\mathbb{N}$ be the length function with respect to $S$ and $\leq$ the Bruhat order. 

For $I\subseteq S$, let $W_I=\langle s\mid s\in I\rangle$ be the parabolic subgroup generated by $I$.  We say 
$I\subseteq S$ is {\it finitary} if $W_I$ is finite, and  in this case we denote by $w_I$ the longest element in $W_I$.
Let
$$\La=\Lambda(W)=\{I\subseteq S\mid I\text{ is finitary}\}\;\;\text{ and }\;\;\La^*=\La\setminus\emptyset,$$
where $\emptyset\in\La$ is the empty set.

For $I\in\La$, we will also denote by $D_I:={D}_{W_I}$ the set of all
distinguished (or shortest) representatives of the right
cosets of $W_I$ in $W$. 
Let
${D}_{IJ}={D}_I\cap {D}^{-1}_J$, where $I,J\in\La$.
Then ${D}_{IJ}$ is the set of shortest
$W_I$-$W_J$ double coset representatives. Symmetrically, let ${D}_{IJ}^+$ be the set of longest
 representatives of $W_I$-$W_J$ double cosets.

For
$d\in{D}_{IJ}$, the subgroup $W_I^d\cap
W_J=d^{-1}W_I d\cap W_J$ is a parabolic subgroup associated
with an element in $\La$,  which will be denoted by $I^d\cap J\in\La$. In other
words, we define
\begin{equation}\label{ladmu}
W_{I^d\cap J}=W_I^d\cap W_J.
\end{equation}
Moreover, the map
$$W_I\times(D_{I^d\cap J}\cap W_J)\longrightarrow W_IdW_J, \quad (w,y)\longmapsto wdy$$
is a bijection.

The Hecke algebra $H_q=H_q(W)$ corresponding to
$W$ is a free $\mathbb{Z}[q]$-module, with basis $\{T_w\mid w\in
W\}$ and generators $T_s$, $s\in S$, subject to the relations (see, e.g., \cite[Ch.~7]{DDPW}): 
\begin{itemize}
\item[(H1)] $T_s^2=(q-1)T_s+q$ for any $s\in S$;
\item[(H2)] $\underbrace{T_sT_tT_s\dots}_{m_{s,t}}=\underbrace{T_tT_sT_t\dots}_{m_{s,t}}$, where $s\neq t$ and $m_{s,t}$ is the order of $st$ in $W$.
\end{itemize}  
Here $T_w=T_{s_1}T_{s_2}\cdots T_{s_l}$ if $w=s_1s_2\cdots s_l$ is a reduced expression.
The relations (H2) are called  braid relations associated with $W$ and we call the relations (H1) and (H2) {\it Hecke  relations}. Note that $H_q$ admits an anti-automorphism
$$\iota:H_q\longrightarrow H_q,\;\; T_w\longmapsto T_{w^{-1}}.$$

For $y,w\in W$, let $P_{y,w}\in\mathbb{Z}[q]$ be the associated Kazhdan--Lusztig polynomial as defined in \cite{KL}. Then, $P_{w,w}=1$, $P_{y,w}=0$ unless $y\leq w$, and $P_{y,w}$ for $y< w$ has degree $\leq \frac12(\ell(w)-\ell(y)-1)$.
Following \cite{KL} (cf. \cite[p.189]{DPS}), let
\begin{equation}\label{C+}
C_w^+=\sum_{y\leq w}P_{y,w}T_y.
\end{equation}
Then $\{C_w^+\}_{w\in W}$ is a new basis for $H_q$ and the following multiplication formula holds:
\begin{equation}\label{CsCw}
C_s^+C_w^+=\begin{cases}(q+1)C_w^+,&\text{ if }sw<w,\\
C_{sw}^++q\sum_{y<w, sy<y}\mu(y,w)q^{\frac12(\ell(w)-\ell(y)-1)}C_y^+, &\text{ otherwise},\end{cases}
\end{equation}
where $\mu(y,w)q^{\frac12(\ell(w)-\ell(y)-1)}$ is the leading term of $P_{y,w}$ if $\mu(y,w)\neq0$. 

For each $w\in W$, fix a reduced expression $\underline{w}=(s_1,s_2\ldots,s_l)$ for $w$ (thus, $w=s_1s_2\cdots s_l$ and $\ell(w)=l$). Let 
\begin{equation}\label{C+2}
C^+_{\underline{w}}=C_{s_1}^+C_{s_2}^+\cdots C_{s_l}^+.
\end{equation} 
Then \eqref{CsCw} implies that there exist $a_{\underline{y},\underline{w}}\in\mathbb Z[q]$ such that
\begin{equation}\label{C+3}
C_w^+=C^+_{\underline{w}}+q\sum_{\underline{y}<\underline{w}}a_{\underline{y},\underline{w}}C^+_{\underline{y}},
\end{equation}
where $\underline{y}<\underline{w}$ denotes a subsequence of $\underline{w}$ obtained by removing some terms in the sequence.


For each $I\in \La$, write  
\begin{equation}\label{x_I}
x_I:=C^+_{w_I}=\sum_{w\in W_I}T_w.
\end{equation}

\begin{definition}\label{S(m|n,r)} The $\mathbb{Z}[q]$-algebra $E_q=E_q(W)=\End_{H_q}(\bigoplus_{I\in\La}x_IH_q)$
is called the {\it Hecke endomorphism algebra} associated to the Coxeter group $W$. 
Specialising $q$ to $0$, the resulting $\mathbb Z$-algebra $E_0=E_q\otimes\mathbb Z$
is called the {\it 0-Hecke endomorphism algebra} associated with $W$.
\end{definition}

Let
 $$x_{s}=C_s^+=T_s+1.$$ The elements $x_{s},s\in S$, generate $H_q$. We now describe a presentation for $H_q$ with these {\it Kazhdan--Lusztig generators}.

For $\{s,t\}\in\La$ and $m\leq m_{s,t}$, let 
$$T_{[m]t}=\underbrace{\cdots T_tT_sT_t}_{m}\text{ and }T_{t[m]}=\underbrace{ T_tT_sT_t\cdots}_{m},$$
and let
$$x_{(s, t)}^{(m)}=\begin{cases}1+T_t,&\mbox{ if } m=1;\\
1+\sum_{i=1}^{m-1}(T_{[i]s}+T_{[i]t})+T_{[m]t},& \mbox{ if } m\geq 2.
\end{cases}$$

Since the generating relations for any Hecke-algebra are determined by the generating relations of 
its rank 2 subalgebras, we have the following. 

\begin{lemma}\label{NewpreHecke}
The Hecke algebra $H_q$ can be presented by generators $x_{s},s\in S$ with
\begin{itemize}
\item[(a)] {\sf Quasi-idempotent relations:}
$(x_{s})^{2}=(q+1)x_{s}$ and 
\item[(b)] {\sf Braid relations:} $ x_{(s,t)}^{(m_{s,t})}=x_{(t,s)}^{(m_{s,t})}$
for all $s,t\in S.$
\end{itemize}
\end{lemma}

We now write the elements $ x_{(s,t)}^{(m)}$ $(m\leq m_{s,t})$ as polynomials in $x_{r}$. Assume first that $\{s,t\}=\{1,2\}$ generates an {\it infinite} dihedral subgroup in the next two lemmas. Thus, the condition $m\leq m_{s,t}$ can be dropped.

\begin{lemma} \label{lastthm} We have  the following recursive formulas for $x_{( 2, 1)}^{(m)}$.
$$
 x^{(m)}_{(2, 1)}=
\left\{
\begin{tabular}{ll} $x_{2}x_{1}$, & if $m=2$;\\
$x_{1}x_{2}x_{1}-qx_{1}$, & if  $m=3$;\\
$(x_{2}x_{1})^2-2qx_{2}x_{1}$, &  if   $m=4$;\\
$x_{(2, 1)}^{(m-2)}(x_{2}x_{1}-2q)-q^2x_{(2, 1)}^{(m-4)}$, & if  $m\geq 5$. \\
\end{tabular}
\right.
$$
\end{lemma}  

\begin{proof} 
Let $$S_{[m]1}=\sum_{i=1}^{m}T_{[i]1} \; \mbox{  and  } \; S_{[m]2}=\sum_{i=1}^{m}T_{[i]2}.$$ Then 
$$
x_{(2, 1)}^{(m)}=
1+S_{[m-1]2}+ S_{[m]1}.
$$
Direct computation gives $x^{(2)}_{(2, 1)}$, $x^{(3)}_{(2, 1)}$, $x^{(4)}_{(2, 1)}$ and, for $m\geq3$,
$$\aligned x^{(m)}_{(2, 1)}x_{2}x_{1}&=x^{(m)}_{(2, 1)}(T_{2}T_{1}+T_2+T_1+1) \\
&= x^{(m+2)}_{(2, 1)} + 2q x^{(m)}_{(2, 1)}+q^2x_{(2, 1)}^{(m-2)}.
\endaligned$$
Hence, the required formula follows.
\end{proof}

By swapping the indices $1$ and $2$, we obtain 
an analogous version of Lemma \ref{lastthm} for $x_{(1, 2)}^{(m)}$.
Consequently, $x_{(2, 1)}^{(m)}$ and $x_{(1, 2)}^{(m)}$ are polynomials in $x_{1}$ and $x_{2}$ without the constant term.
Next, based on Lemma \ref{lastthm}, we will deduce an explicit formula for $x^{(m)}_{(2, 1)}$ as a linear combination of monomials 
in $x_{1}$ and $x_{2}$,

\begin{equation}\label{equation}
x^{(m)}_{(2, 1)}=
\sum_{j=1}^{m} b^m_{j}x_{[j]1},
\end{equation}
where $b^m_s\in\mathbb Z[q]$ and  $x_{[m]1}=\underbrace{\dots x_{2}x_{1}}_m$. Similarly, 
$x_{[m]2}=\underbrace{\dots x_{1}x_{2}}_m$.

\begin{lemma} \label{lastlemma}
The coefficient  $b^m_j $ is equal to $\left(\begin{matrix} j+s-1\\j-1 \end{matrix}\right)(-q)^s$ if  $s=\frac{m-j}{2}$ is an integer, and $0$ otherwise. Consequently, 
$$
x^{(m)}_{(2, 1)}=\sum_{i=0}^{\lfloor \frac{m-1}{2}\rfloor}{m-i-1\choose i}(-q)^{i}x_{[m-2i]1}.
$$
\end{lemma}
\begin{proof} Use  induction on $m$.
Observe  from Lemma \ref{lastthm} that $x^{(m)}_{(2, 1)}$ has degree $m$ and is a linear combination of monomials 
with degrees of  the same parity as $m$. So $$b^m_j=0 \mbox{ if } \frac{m-j}{2}\not\in \mathbb{Z}.$$ 

Further,  by Lemma \ref{lastthm} and (\ref{equation}), $b^m_{m}=1$, $b^{m}_{m-2}=-(m-2)q$
and $b^{m}_1=(-q)^{\frac{m-1}{2}}$ if $m$ is odd (resp., $b^{m}_2=\frac m2(-q)^{\frac{m-2}{2}}$ if $m$ is even). In general, for $3\leq j\leq m-4$, let $s=\frac{m-j}{2}$, 
where $j$ and $m$ have the same parity. By Lemma \ref{lastthm}
and induction,  
\begin{equation*}
\begin{split}
b^{m}_j  &=  b^{m-2}_{j-2}-2qb^{m-2}_j-q^2b^{m-4}_j\\
&=(-q)^s \left(\begin{matrix} j+s-3\\j-3 \end{matrix}\right)-2q (-q)^{s-1} \left(\begin{matrix} j+s-2\\j-1 \end{matrix}\right)-
q^2 (-q)^{s-2}\left(\begin{matrix} j+s-3\\j-1 \end{matrix}\right)\\
&=(-q)^s \left(\begin{matrix} j+s-1\\j-1 \end{matrix}\right),
\end{split}
\end{equation*}
as required.
\end{proof}

\begin{remark}
The coefficients $b^m_{j}$ also satisfies the recursive formula 
$$b_j^m=b_{j-1}^{m-1}-qb_{j}^{m-2}.$$
This formula evaluated at $q=1$ gives Elias' recursive formula in \cite[Def.~2.10]{Elias}
$$d_m^j=d^{j-1}_{m-1}-d^j_{m-2},$$   
where $d_m^j=b^m_j|_{q=1}$. So we can view $b^m_j$ as a $q$-analogue of $d^j_m$.
\end{remark}

By definition, for $m< m_{s,t}<\infty$, the elements $ x_{(s,t)}^{(m)}$ and $x_{(t,s)}^{(m)}$ are distinct. However, if $m= m_{s,t}$, then $ x_{(s,t)}^{(m_{s,t})}=x_{(t,s)}^{(m_{s,t})}=x_{\{s,t\}}$ are two distinct expressions in the Kazhdan--Lusztig generators.

In general, we define \begin{equation}\label{sqsubset}
I\sqsubset J\iff I\subset J, |I|=|J|-1. 
\end{equation} 
Then for each $I\in\Lambda$, every total ordering on $I$: 
\begin{equation}\label{I.}
I_\bullet:\emptyset=I_0\sqsubset I_1\sqsubset I_2\sqsubset \dots \sqsubset I_m=I,
\end{equation}
 gives an expression for $x_I$ as follows. 
Fix a reduced expression $w_{I_\bullet}=s_ls_{l-1}\cdots s_1$ such that $w_{I_i}=s_{j_i}s_{j_i-1}\cdots s_1$ for all $i=1,\ldots,m$ and let
$\underline{w}_{I_\bullet}=(s_l,s_{l-1},\cdots ,s_1)$. 
Let
$$x_{\underline{w}_{I_\bullet}}=x_{s_l}x_{s_{l-1}}\cdots x_{s_1}.$$
 By \eqref{C+3}, we have the following.

\begin{lemma} \label{x_I}For any $I\in\Lambda$, any ordering $I_\bullet$ as in \eqref{I.}, and any reduced expression $\underline{w}_{I_\bullet}$, there exist polynomials $a_{\underline{y},\underline{w}_{I_\bullet}}\in\mathbb Z[q]$ such that
$$x_I=x_{\underline{w}_{I_\bullet}}+q\sum_{\underline{y}<\underline{w}_{I_\bullet}}a_{\underline{y},\underline{w}_{I_\bullet}}x_{\underline{y}}.$$
\end{lemma}


\section{Two preliminary lemmas}
{\it In this section, $R$ can be any commutative ring with $1$.} We prove two technical lemmas.  
Recall that a quiver $Q$ is an oriented graph, consisting of a set $Q_0$ of vertices and a set $Q_1$ of 
arrows connecting the vertices. The path algebra $RQ$ of $Q$ is the free $R$-module with basis the set 
of (oriented) paths in $Q$ and multiplication given by concatenation of paths, i.e. for any two paths 
$\alpha$ and $\beta$, 
$\alpha\cdot \beta=\alpha\beta $ if the starting vertex of $\alpha$ is the ending vertex of $\beta$ and 
$0$ otherwise. 
For any vertex $i$, we denote by $e_i$ the trivial path at $i$. A relation in $RQ$ is an $R$-linear combination of 
paths with a fixed starting vertex and a fixed ending vertex.  A set of relations $\sI$ defines a quotient algebra 
$RQ/\langle \sI\rangle$, where $\langle \sI\rangle$ is the ideal in $RQ$ generated by $\sI$.

\begin{lemma}\label{techlemma2}
Let $(Q, \sI)$ be a quiver $Q$ with relations $\sI$ and let $(Q', \sI')$ be the 
quiver obtained from $Q$ by adding a loop $\alpha$ at a vertex to $Q_1$ and a relation $\alpha - p$ to $\sI$,
where $p\in RQ$ is a cycle through the vertex. 
Then the inclusion of $Q$ into $Q'$ induces an $R$-algebra isomorphism 
$$RQ/\langle\sI\rangle \rightarrow RQ'/\langle\sI'\rangle. $$
\end{lemma} 

\begin{proof} 
Since $RQ\cap\langle\sI'\rangle=\langle\sI\rangle$, the homomorphism $RQ\subset RQ'\to RQ'/\langle\sI'\rangle$ induces a monomorphism
$\phi:RQ/\langle\sI\rangle\to RQ'/\langle\sI'\rangle$. Now every coset in $RQ'/\langle\sI'\rangle$ has a representative in $RQ$. Hence,
this monomorphism is an isomorphism (and $\phi^{-1}$ is the map sending $x$ to $x$ and $\alpha$ to $p$, where $x$
is an arbitrary path in $RQ$).
\end{proof}

Let $A$ be an associative  $R$-algebra defined by generators $a_1$, $\dots$, $a_n$ and a set $\sI$ of relations. Let 
$p_0=1, p_1, \dots, p_s$ be idempotents in $A$. Note that each $p_i$ ($i\neq0$) or $p\in\sI$ is a (noncommutative) polynomial 
$p_i(a_1, \dots, a_n)$ or $p(a_1,\ldots,a_n)$ in $a_1, \dots, a_n$. Let 
$$
B=\End_A\bigg(\bigoplus_{i=0}^sp_iA\bigg)=\bigoplus_{0\leq i,j\leq s}\Hom_A(p_iA,p_jA).
$$
Denote by $1_i$ the identity map on $p_iA$, $u_i: p_iA\rightarrow A$ the natural embedding and by $d_i: A\rightarrow p_iA$ the natural 
surjection, given by left multiplication with $p_i$. We also view $a_i$ as an element in $B$, which defines 
the map $A\rightarrow A$ by left multiplication with $a_i$. All $1_i$, $u_i,d_i$ and $a_i$ are regarded as elements of $B$ by sending other components $p_jA$ to 0.


Let $Q'$ be the quiver with vertices $0, 1,\dots, s$, loops $\alpha_i$ at vertex $0$ for $1\leq i \leq n$, arrows $\upsilon_i$
from vertex $i$ to $0$ and $\delta_i$ from $0$ to $i$ for $1\leq i\leq s$. Let $\sJ'\subseteq RQ'$ be the set of elements obtained from 
\begin{itemize}
\item[(1)] $p':=p(\alpha_1, \dots, \alpha_s)$ for all $p\in \sI$; 
\item[(2)] $\delta_i\upsilon_i=e_i$ for  all 
$1\leq i\leq s$;  
\item[(3)] $\upsilon_i\delta_i=p_i':=p_i(\alpha_1, \dots, \alpha_s)$ for  all 
$1\leq i\leq s$. 
\end{itemize}
Let $R\langle \alpha_1, \dots, \alpha_s\rangle$ be the (centraliser) subalgebra generated by $\alpha_1, \dots, \alpha_s$ with the identity $e_0$ 
and let
$$C=RQ'/\langle\sJ'\rangle.$$
Since $R\langle e_0,  \alpha_1, \dots, \alpha_s\rangle\cap \langle\sJ'\rangle$ is the ideal generated by $p'$ for $p\in\sI$, it follows that the subalgebra $C_0$ of $C$ generated by all $\alpha_i+\langle\sJ'\rangle$ is isomorphic to $A$. Thus, as an image of $p_i$, each $p_i'$ is an idempotent in $C$.


\begin{lemma}\label{techlemma1}
The algebra $B$ with the identity $1=\sum_i 1_i$ is generated by  $a_i$, $u_j$ and $d_j$  subject to the relations
\begin{itemize}
\item[(1)] $p(a_1, \dots, a_n)$ for all $p\in \sI$;
\item[(2)] $d_iu_i=1_i$ and $u_id_i=p_i$ for all $1\leq i\leq s$;
\item[(3)] $u_i1_i=u_i=1_0u_i$ and $d_i1_0=d_i=1_id_i$ for all $1\leq i\leq s$; 
\item[(4)] $1_i1_j=0$ for $i\not= j$ and $1_i^2=1_i$ for all $i$.
\end{itemize}
\end{lemma}

\begin{proof}
Since $\Hom_A(p_iA, p_jA)\cong p_jAp_i$, the algebra $B$ is generated by  $a_i$, $u_j$ and $d_j$.
We show that there is an  isomorphism from $C$ to $B$, sending $\alpha_i, \, \upsilon_j, \, \delta_j$ to 
$a_i, \, u_j, d_j$, respectively, and thus conclude that the relations are the generating relations of 
$B$. 
 
First note that the identity $1_B$ in $B$ is 
$$1_B=1_0+1_1+\dots +1_s$$ and
$$
1_0B1_0\cong \Hom_A(A, A)\cong A.
$$

Define
$$\phi: C \rightarrow B$$
with 
$\alpha_i \mapsto a_i, \,e_i\mapsto 1_i, \, \upsilon_i\mapsto u_i$ and $\delta_i\mapsto d_i$. 
By definition, $a_j, u_i$ and $d_i$ satisfy the  relations 
$\sJ'$. 
So $\phi$ is a well-defined  surjective homomorphism.
 In particular, it induces a surjective 
map $\phi_{ij}: e_i Ce_j \rightarrow p_iA p_j(\cong 1_iB1_j)$.  We will show that all the $\phi_{ij}$
are linear isomorphisms and thus conclude that $\phi$ is an isomorphism.
First, by comparing the generating relations and noting $C_0=e_0Ce_0$ , we have that the restriction of $\phi$ to $e_0Ce_0$ 
$$\phi|_{e_0Ce_0}:  e_0 C e_0\rightarrow A$$
is an isomorphism. In particular, it induces an isomorphism from $p_i'Cp_j'$ to $p_iAp_j$.

Using the relations between $\upsilon_i, \delta_i, p_i$ and $e_i$, 
we have
$$
e_iCe_j=\delta_i\upsilon_iC \delta_j\upsilon_j\subseteq \delta_i C \upsilon_j=e_i \delta_i C\upsilon_j e_j\subseteq e_iCe_j.
$$
Thus $$e_iCe_j=\delta_iC\upsilon_j.$$
Similarly,
$$p_i' C p_j' \cong
\delta_ip_i' C p_j'\upsilon_j = \delta_iC\upsilon_j=e_iCe_j,
$$
where    the (linear) isomorphism, denoted by $\psi$, is 
given by left multiplication with $\delta_i$ and right multiplication with $\upsilon_j$.
We have the following commutative diagram
$$
\xymatrix{
p_i'Cp_j'\ar[r]^{\epsilon} \ar[d]_{\psi}& p_iAp_j\ar[d]^{=}\\
e_iCe_j \ar[r]^{\phi_{ij}}& p_iAp_j,
}
$$
where $\epsilon$ is an isomorphism induced by the isomorphism $\phi|_{e_0Ce_0}$.
Thus
 $\phi_{ij}$ is a linear isomorphism, as required.
\end{proof}
\begin{remark}
The corresponding relations of Lemma \ref{techlemma1} (3) and (4) in $C$ are implicit  in the definition of the multiplication 
in the path algebra.
\end{remark}
\section{A quiver approach to Hecke endomorphism  algebras}\label{section4}

{\it In the rest of the paper, let 
\begin{equation}\label{ring R}
R=\bigg\{\frac{f}{1+qg}\;\bigg|\; f, g\in \mathbb{Z}[q]\bigg\}=\mathbb Z[q]_P,
\end{equation}
be the ring obtained by localising $\mathbb{Z}[q]$ at the multiplicative set $P$ of all polynomials with the constant term 1. }

We now apply the technical lemmas to give presentations of the Hecke endomorphism  algebras over $R$. 
Then applying the functors  $-\otimes _{\mathbb{Z}[q]}\mathbb{Q}(q)$ and 
$-\otimes _{\mathbb{Z}[q]}\mathbb{Z}$ gives the presentations of the algebra over $\mathbb{Q}(q)$ and
$\mathbb{Z}$, respectively.





By extending the ground ring to $R$, we write 
$$RH_q=R\otimes_{\mathbb{Z}[q]}H_q.$$
The elements in $RH_q$ are $R$-linear combinations 
$\sum_{w\in W}r_wT_w$, where $r_w\in R$.
Recall that  the Poincar$\acute{e}$ polynomial associated to $W_I$  is $$\pi(I)=\sum_{w\in W_I}q^{\ell(w)}.$$
Note that the elements $x_I$ defined in \eqref{x_I} satisfies $x_I^2=\pi(I)x_I$. Since $\frac1{\pi(I)}\in R$, the elements
$$p_I=\frac{x_I}{\pi(I)}$$
are idempotents in $RH_q$ and $p_IRH_q=x_IRH_q$.

{\it In the sequel, the notation $p_I|_{\{t_s\}_{s\in S}}$ denotes the element in an $R$-algebra $A$ obtained by replacing $T_s$ in $p_I$ with $t_s\in A$ for all $s\in S$.}

We now apply Lemma \ref{techlemma1} to the Hecke algebra $A=RH_q$, the  
generators $T_s,\;s\in S$, the set $\sI$ of Hecke  relations  (cf. \cite{JS}), and the idempotents
$p_I,\; I \in \Lambda$, 
with $p_0=p_{\emptyset}=1$, and to give the first 
presentation of the Hecke endomorphism  algebras $B=RE_q$ over $R$, where
$$RE_q=R\otimes _{\mathbb{Z}[q]}E_q.$$

Let $(Q', \sJ')$ be the quiver associated with $RH_q(W)$ as constructed in the proof of Lemma \ref{techlemma1}. Thus, the vertices are $e_I$ with $I\in \Lambda$, 
arrows are $\alpha_s$, $\upsilon_I$, $\delta_I$ with $s\in S$ and $I\in \Lambda^*$,
and $\sJ'$ consists of
\begin{itemize}
\item[(J$'$1)] Hecke relations on $\alpha_s,s\in S$;
\item[(J$'$2)] $\delta_I\upsilon_I=e_I$ for all $I\in\Lambda^*$;
\item[(J$'$3)] $\upsilon_I\delta_I=p_I':=p_I|_{\{\alpha_s\}_{s\in S}}$ for all $I\in\Lambda^*$.
\end{itemize}

In the case $W=S_4$, the symmetric group on 4 letters, and $S=\{s_i\mid i=1,2,3\}$ with $s_i=(i,i+1)$. The quiver $Q'$, where $s_i$ is identified with $i$, has the form
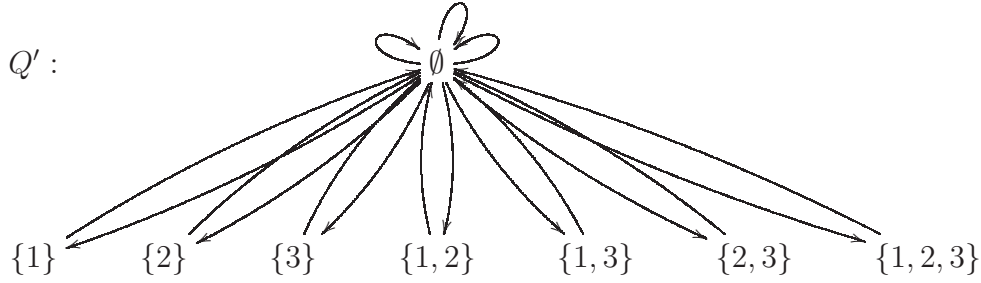
\vspace{-2ex}
$$ 
\xymatrix{\\Q': &&&\emptyset \ar@(l, ul)[]  \ar@(u,  ur)[] \ar@(r, ur)[]
\ar@/^/[ddlll] \ar@/^/[ddll] \ar@/^/[ddl] \ar@/^/[dd] \ar@/_/[ddr] \ar@/_/[ddrr]\ar@/_/[ddrrr]
\\\\
\{1\} \ar@/^/[uurrr]& \{2\} \ar@/^/[uurr]&\{3\} \ar@/^/[uur]&\{1, 2\}\ar@/^/[uu] 
&\{1, 3\}\ar@/_/[uul] &\{2, 3\}\ar@/_/[uull]  &\{1, 2, 3\} \ar@/_/[uulll] 
}$$

\vspace{.3cm}
Note that, in $RH_q$, $x_{\{s\}}=x_s=1+T_s$ and  $x_I x_J=\pi(I)x_J$ for $I\subseteq J$.
\begin{lemma}\label{cor} 
We have the following relations
in $C=RQ'/\langle\sJ'\rangle$.
\begin{itemize}
\item[(1)]  $\alpha_{\{s\}}=(1+q)\upsilon_{\{s\}}\delta_{\{s\}}-e_\emptyset$ 
for all $s\in S$.
\item[(2)] For any $I$ and $J$ in $\Lambda^*$  with  $I\subset J$ , 
$p_I'=\upsilon_I\delta_I$ and $p_I'p_J'=p_J'=p_J'p_I'$.
\end{itemize}
\end{lemma}
\begin{proof} (1) follows from (J$'$3) by taking $I=\{s\}$. 
 (2) is true, by the defintion of $p'_I, p'_j$,  the facts that $x_Ix_J=\pi(I)x_J=x_Jx_I$ and $\alpha_s$ $(s\in S)$ satisfy the Hecke relations. 
\end{proof}

We can now present the Hecke endomorphism  algebra $RE_q$ via the quiver $Q'$.

\begin{proposition}\label{pres1}
The Hecke endomorphism  algebra  $RE_q(W)$ associated with a Coxeter system $(W,S)$ is isomorphic to the algebra
$RQ'/\langle\sJ'\rangle$.
 More precisely, $RE_q(W)$ is generated by 
$$T_s, 1_I, u_J, d_J\quad(s\in S, \;I\in\Lambda, \;J\in \Lambda^*),$$ subject to  the relations 
\begin{itemize}
\item[(1)] Hecke  relations on $T_s$, $s\in S$;
\item[(2)] $d_I u_I=1_I$ and $u_Id_I=p_I$, $I\in \Lambda^*$;
\item[(3)] $u_I 1_I=u_I=1_\emptyset u_I$ and $d_I 1_\emptyset=d_I=1_Id_I$, $I\in \Lambda^*$;
\item[(4)] $1_I1_J=0$ for $I\not= J$ in $\Lambda$ and $1_I^2=1_I$ for all $I\in \Lambda$.
\end{itemize}
\end{proposition}

\begin{proof}
It follows from Lemma \ref{techlemma1} with the isomorphism 
$$\phi: RQ'/\sJ'\rightarrow RE_q(W)$$
given by $\alpha_s\mapsto T_s$, $e_I\mapsto 1_I$, $\upsilon_I\mapsto u_I$ and $\delta_I\mapsto d_I$.
\end{proof}


Let $Q$ be the quiver obtained from $Q'$ by removing the loops $\alpha_s, s\in S$, and let 
\begin{equation}\label{tau_s}
\tau_s=(1+q)\upsilon_{\{s\}}\delta_{\{s\}}-e_\emptyset\text{ for all }s\in S.
\end{equation}
Then $\tau_s\in RQ$ for all $s\in S$. 
Define 
\begin{equation}\label{sJ}
\sJ=\{\mbox{Braid  relations on }\tau_s, s\in S\}\cup\{\upsilon_I\delta_I=p_I|_{\{\tau_s\}_{s\in S}}, \delta_I\upsilon_I=e_I\mid I\in \Lambda^*\}.
\end{equation}

\begin{proposition}\label{pres2}The inclusion of $Q$ into $Q'$ induces an isomorphism of $R$-algebras 
 $$RQ/\langle\sJ\rangle \rightarrow RQ'/\langle\sJ'\rangle.$$ 
In particular, the Hecke endomorphism  algebra $RE_q(W)$ is isomorphic to the path algebra $RQ$ modulo the relations in $\sJ$.
\end{proposition}

\begin{proof}
 We show that $\tau_s$ satisfies the quadratic Hecke relations.  
$$\aligned
\tau_{s}^2&=(1+q)^2(\upsilon_{\{s\}}\delta_{\{s\}})^2-2(1+q)\upsilon_{\{s\}}\delta_{\{s\}}+e_\emptyset\\
&=(1+2q+q^2)\upsilon_{\{s\}}e_{\{s\}}\delta_{\{s\}}-2(1+q)\upsilon_{\{s\}}\delta_{\{s\}}+e_\emptyset\\
&=(q^2-1)\upsilon_{\{s\}}\delta_{\{s\}}+e_\emptyset\\
&=q(q+1)\upsilon_{\{s\}}\delta_{\{s\}}-(q+1)\upsilon_{\{s\}}\delta_{\{s\}}+e_\emptyset\\
&=(q-1)\tau_s +qe_\emptyset,
\endaligned$$
Now the proposition follows from Lemma \ref{techlemma2}.
\end{proof}

We now use the presentation above for $RE_q$ to derive a presentation in terms of  modified Williamson's generators.

For any $I\subset J$ in $\Lambda$, define
$$\begin{cases}
\upsilon_{\emptyset, J}=\upsilon_J,\;\;\; \delta_{J, \emptyset}=\delta_J, &\text{ if }I=\emptyset;\\
\upsilon_{I,J}=\delta_I\upsilon_J,\;\;\; \delta_{J,I}=\delta_J\upsilon_I, &\text{ if }I\neq\emptyset.\\
\end{cases}
$$

\begin{lemma}\label{newgenerators}
Suppose that $I\subset J\subset K$ in $\Lambda$. Then we have the following in $RQ/\langle\sJ\rangle$.
\begin{itemize}\item[(1)]
$\upsilon_{I, K}=\upsilon_{I, J}\upsilon_{J, K} \;\;\mbox{ and }\;\; \delta_{K, I}=\delta_{K, J}\delta_{J, I}.
$
In particular, $$\upsilon_K=\upsilon_J\upsilon_{J, K} \mbox{ and } \delta_{K}=\delta_{K, J}\delta_{J}.$$ 

\item[(2)] $\delta_{J, I}\upsilon_{I, J}=e_{J}$ and thus $\delta_{J, I}\upsilon_{I, J}$ is an idempotent.
\item[(3)] $\upsilon_{I, J}\delta_{J, I}=\delta_I p_J \upsilon_I=p_Je_I$, where $p_J=p_J|_{\{\tau_s\}_{s\in S}}$.
\end{itemize}
\end{lemma}
\begin{proof} (1)
By defintion, 
$\upsilon_{I, J}\upsilon_{J, K}=\delta_I\upsilon_J\delta_J\upsilon_K$.
Following the relations $\delta_I\upsilon_I=e_I$ and $\upsilon_I\delta_I=p_I$ for any $I\in\La^*$ in \eqref{sJ} and Lemma \ref{cor}, 
$$
\upsilon_{I, J}\upsilon_{J, K}=\delta_I\upsilon_I(\delta_I\upsilon_J\delta_J\upsilon_K) \delta_K\upsilon_K=\delta_I p_Ip_Jp_K\upsilon_K=\delta_Ip_K\upsilon_K=\delta_I\upsilon_K\delta_K\upsilon_K=\delta_I\upsilon_K,
$$ which is $\upsilon_{I, K}$.
So the first identity follows. Similarly, the second one holds.

(2) Also, 
$$
\delta_{J, I}\upsilon_{I, J}=e_J\delta_{J, I}\upsilon_{I, J}e_J
=\delta_{J}\upsilon_J (\delta_{J}\upsilon_{I}\delta_{I}\upsilon_{J}) \delta_J\upsilon_J
=\delta_Jp_Jp_Ip_J\upsilon_J=\delta_Jp_J\upsilon_J=\delta_J\upsilon_J\delta_J\upsilon_J,
$$ which is $e_J^2=e_J$, as required.

(3) It follows from (1) and Lemma \ref{cor}.
\end{proof}
Lemma \ref{newgenerators} gives the following sandwich relations.
\begin{corollary}
If $I\subset J, J'\subset K$, then  we have 
$$\upsilon_{I, J}\upsilon_{J, K}=\upsilon_{I, J'}\upsilon_{J', K},\;\; \;\delta_{K,J}\delta_{J, I}=\delta_{K,J'}\delta_{J', I}.$$
\end{corollary}

We now have the following quiver presentation for the Hecke endomorphism algebra $RE_q(W)\cong RQ/\langle\sJ\rangle$. Recall the relation $I\sqsubset J$ defined in \eqref{sqsubset}.

\begin{corollary}\label{newgenerators2}
The algebra $RQ/\langle\sJ\rangle$ is generated by $e_K$, $\upsilon_{I, J}$ and $\delta_{J, I}$ 
for $I, J, K\in \Lambda$ with $I\sqsubset J$, 
and the ideal $\langle \sJ\rangle $ is generated by the following relations. 
\begin{itemize}
\item[(J1)] Braid relations on $\{\tau_s \mid s\in S\}$;
\item[(J2)] $\delta_{J, I}\upsilon_{I, J}=e_{J}$ for all $I, J\in \Lambda$ with $I\sqsubset J$;
\item[(J3)] $\upsilon_{\emptyset, I_{1}}\dots \upsilon_{I_{m-1}, I}\delta_{I, I_{m-1}}\dots \delta_{I_2, I_1}\delta_{I_1, \emptyset}=
p_{I}|_{\{\tau_s\}_{s\in S}}$ for all $I\in \La^*$, and sequences
$$\emptyset=I_0\sqsubset I_1\sqsubset I_2\sqsubset\dots \sqsubset I_m=I.$$
\item[(J4)] $\upsilon_{I, J}\upsilon_{J, K}=\upsilon_{I, J'}\upsilon_{J', K}$ and $\delta_{K, J}\delta_{J, I}=\delta_{K, J'}\delta_{J', I}$ for all $I,J,J',K\in\Lambda$ with
$I\sqsubset J\sqsubset K$ and $I\sqsubset J'\sqsubset K.$
\end{itemize} 
\end{corollary}
\begin{proof}
It follows from Lemma \ref{newgenerators} and the definition of $\sJ$ in \eqref{sJ}. 
\end{proof}

Note that (J3) and (J4) imply that $\upsilon_{\emptyset, I_{1}}\dots \upsilon_{I_{m-1}, I}\delta_{I, I_{m-1}'}\dots \delta_{I_2', I_1'}\delta_{I_1', \emptyset}=
p_{I}|_{\{\tau_s\}_{s\in S}}$ for all $I\in \La$ and sequences
$$\emptyset=I_0\sqsubset I_1\sqsubset I_2\sqsubset\dots \sqsubset I_m=I\text{ and }
\emptyset=I_0'\sqsubset I_1'\sqsubset I_2'\sqsubset\dots \sqsubset I_m'=I
.$$

\section{The Hasse quiver presentation}
Following Proposition \ref{pres2}  and Corollary \ref{newgenerators2}, 
we now give a new presentation for $RE_q(W)$ in term of the Hasse quiver of the Coxeter system $(W,S)$.

Recall the set $\La$ of finitary subsets of $S$. With the inclusion relation, $\La$ is a poset. If we replace every edge in the Hasse diagram of the poset $\La$ by a pair of arrows in opposite direction, the resulting quiver $\tilde{Q}=\tilde Q(W)$ is called the {\it Hasse quiver}
associated to $W$.  

Note that the new quiver $\tilde{Q}$ can be obtained from $Q$ by removing all arrows $\upsilon_I$, $\delta_I$ for $|I|>1$, renaming the arrows $\upsilon_I$, $\delta_I$ for 
$|I|=1$ to $\tilde\upsilon_{\emptyset,I}$, $\tilde\delta_{I,\emptyset}$, and introducing new arrows 
$\tilde\upsilon_{I, J}$ from $J$ to $I$ and $\tilde\delta_{J, I}$ from $I$ to $J$ for all $I,J\in\La^*$ with $I\sqsubset J$. That is,
the quiver has vertex set $\La$ and arrow set
$$\{\tilde\delta_{J,I},\tilde\upsilon_{I,J}\mid I,J\in\Lambda,I\sqsubset J\}.$$


For example,
for  $W=\fS_4$, the symmetric group on 4 letters,
the quiver $\tilde{Q}$ is as follows, 
where the arrows 
between $\emptyset$ and $\{1\}, \{2\}$, $\{3\}$ are the same as those in $Q$.
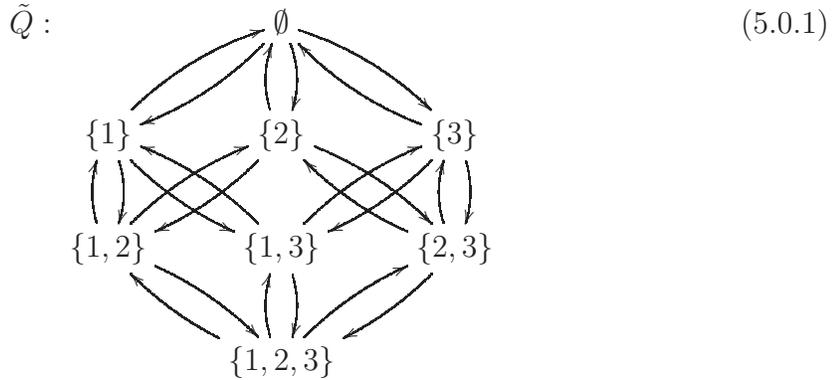
\begin{equation}\label{HasseA3}
\tilde{Q}:
\xymatrix{& \emptyset  \ar@/^/[d]  \ar@/^/[dl] \ar@/^/[dr]&\\
\{1\} \ar@/^/[d] \ar@/_/[dr]  \ar@/^/[ur] &\{2\}\ar@/^/[dr] \ar@/^/[dl]\ar@/^/[u]
&\{3\} \ar@/^/[d]\ar@/^/[dl]  \ar@/^/[ul] \\
\{1, 2\}\ar@/^/[dr]  \ar@/^/[ur]  \ar@/^/[u]&
\{1, 3\} \ar@/^/[d] \ar@/^/[ur] \ar@/_/[ul]&\{2, 3\}\ar@/^/[dl]\ar@/^/[ul]\ar@/^/[u]\\
&\{1, 2, 3\} \ar@/^/[ur]\ar@/^/[ul] \ar@/^/[u]&}
\end{equation}

The path algebra $R\tilde Q$ of $\tilde Q$ over a ring $R$ admits an anti-involution
\begin{equation}\label{tau}
\tau:R\tilde Q\longrightarrow R\tilde Q
\end{equation}
that interchanges $\tilde \upsilon_{I,J}$ and $\tilde \delta_{I,J}$.


Let $\tilde e_K$ denote the trivial path in $\tilde Q$ at the vertex $K\in\La$. For each singleton $\{s\}\in \Lambda$, let 
$$\tilde\chi_{s}=(q+1)\tilde\upsilon_{\emptyset,\{s\}}\tilde\delta_{\{s\},\emptyset}\in R\tilde Q.$$
 For each $I\in\Lambda$, a total ordering $I_\bullet$ on $I$ as in \eqref{I.}, and a reduced expression $\underline{w}_{I_\bullet}$, let
\begin{equation}\label{chi_I}
\tilde\chi_{I_\bullet}=\tilde\chi_{\underline{w}_{I_\bullet}}+q\sum_{\underline{y}<\underline{w}_{I_\bullet}}a_{\underline{y},\underline{w}_{I_\bullet}}\tilde\chi_{\underline{y}}\text{ and }\tilde\rho_{I_\bullet}=\frac{\tilde\chi_{I_\bullet}}{\pi(I)},
\end{equation}
where $\tilde\chi_{\underline{w}}=\tilde\chi_{s_l}\tilde\chi_{s_{l-1}}\cdots\tilde \chi_{s_1}$ if $\underline{w}=(s_j,s_{l-1},\cdots,s_1)$ and $a_{\underline{y},\underline{w}_{I_\bullet}}\in\mathbb Z[q]$ are given in Lemma \ref{x_I}.


\begin{theorem}\label{tildeQ} Maintain the notations $\La$ and $\tilde \rho_{I_\bullet}$, etc., for $W$ as above.
Let $\tilde\sJ$ be the subset of $R\tilde Q$ defined by the following relations 
\begin{itemize}
\item[($\tilde{\text{J}}$1)] {\sf Idempotent relations}: \vspace{-1ex}
$$\tilde\delta_{J, I}\tilde\upsilon_{I, J}=\tilde e_{J}$$ for all $I, J\in \Lambda$ with $I\sqsubset J$;
\item[($\tilde{\text{J}}$2)] {\sf Sandwich relations}: \vspace{-1ex}
$$\tilde\upsilon_{I, J}\tilde\upsilon_{J, K}=\tilde\upsilon_{I, J'}\tilde\upsilon_{J', K}\text{ and }\tilde\delta_{K, J}\tilde\delta_{J, I}=\tilde\delta_{K, J'}\tilde\delta_{J', I}$$ for all $I,J,J',K\in\Lambda$ with
$I\sqsubset J\sqsubset K$ and $I\sqsubset J'\sqsubset K.$
\item[($\tilde{\text{J}}$3)] {\sf Extended braid relations}: \vspace{-1ex}
$$\tilde\upsilon_{\emptyset, I_{1}}\dots \tilde\upsilon_{I_{m-1}, I}\tilde\delta_{I, I_{m-1}}\dots\tilde \delta_{I_2, I_1}\tilde\delta_{I_1, \emptyset}=\tilde\rho_{I_\bullet}$$
 for all $I\in \La^*$, $|I|>1$, and given ordering $I_\bullet:\emptyset=I_0\sqsubset I_1\sqsubset\dots \sqsubset I_m=I.$
\end{itemize}
Then $RE_q(W)\cong RQ/\langle\sJ\rangle\cong R\tilde Q/\langle\tilde\sJ\rangle$.
\end{theorem}

\begin{proof} By ($\tilde {\rm J}$3), for any $s\in S$ we have $\tilde\chi_{\{s\}}=(q+1)\tilde\upsilon_{\emptyset,\{s\}}\tilde\delta_{s,\emptyset}$. Thus, by ($\tilde {\rm J}$1),
we have $\tilde\chi_{\{s\}}^2=(q+1)\tilde\chi_{\{s\}}$. For any $s,t\in S$, ($\tilde {\rm J}$2) and($\tilde {\rm J}$3) imply $\tilde\rho_{(s,t)}=\tilde\rho_{(t,s)}$. Hence,
$\tilde \chi_{(s,t)}^{m_{s,t}}=\tilde \chi_{(t,s)}^{m_{s,t}}$. By Lemma \ref{NewpreHecke}, the subalgebra generated by $\tilde\chi_{s}$ is isomorphic to $RH_q(W)$. Note that this subalgebra is isomorphic to $\tilde e_\emptyset(R\tilde Q/\langle\tilde\sJ\rangle)\tilde e_\emptyset$ by ($\tilde {\rm J}$3).
Thus, $\tilde\rho_{I_\bullet}=\tilde\rho_I$ which is independent of the ordering on $I$. Further, if we put $\tilde \tau_s=\tilde\chi_{s}-1$ for all $s\in S$, then $\tilde\rho_{I_\bullet}=p_I|_{\{\tilde\tau_s\}_{s\in S}}$.

Consider the $R$-algebra homomorphism $\phi:R\tilde Q\to RQ$ sending $\tilde\upsilon_{I,J},\tilde\delta_{J,I},\tilde e_K$ to $\upsilon_{I,J},\delta_{J,I}, e_K$, respectively. This map induces a surjective algebra homomorphism
$\bar \phi:R\tilde Q\to RQ/\langle\sJ\rangle$. 
By Corollary \ref{newgenerators2}, it is clear that $\phi^{-1}(\langle\sJ\rangle)=\langle\tilde\sJ\rangle$ and, hence, $\phi$ induces the required isomorphism.
\end{proof}

The presentation in Theorem \ref{tildeQ} is called the {\it Hasse quiver presentation} of $E_q(W)$.

We now give a direct description of the isomorphism from $R\tilde Q/\langle\tilde\sJ\rangle$ to $RE_q(W)$ and turn the Hasse quiver presentations to a full presentation for $RE_q(W)$.

Note that, if $I\subseteq J$, then $x_J=hx_I=x_I \iota(h)$, where $h=h_{J,I}=\sum_{w^{-1}\in D_I\cap W_J}T_w$ and $\iota$ is the anti-automorphism on $RH_q$ introduced in Seciton 2.  
So  $x_JH_q$ is a submodule of 
$x_IH_q$  and 
$$p_J=\bigg(\frac{\pi(I)}{\pi(J)} h_{J,I} \bigg)p_I.$$  
Denote by $$u_{I, J}:p_J(RH_q)\rightarrow p_I(RH_q) $$
the natural inclusion and by
$$
d_{J, I}: p_I(RH_q)\rightarrow p_J(RH_q) 
$$
the splitting map of the inclusion given by left multiplication with $\frac{\pi(I)}{\pi(J)}h_{J,I}$.  
Define the map 
$$\hbar:R\tilde Q/\langle\tilde\sJ\rangle\to RE_q(W), $$ sending $\tilde\upsilon_{I,J},\tilde\delta_{J,I},\tilde e_K$ to $u_{I, J}$, $d_{J, I}$ and $1_K$, respectively, where $I,J,K\in\La$ with
$I\sqsubset J$.

\begin{corollary}\label{pres3} The map $\hbar$ induces an algebra isomorphism $R\tilde Q/\langle\tilde\sJ\rangle\cong RE_q(W)$. In other words, $RE_q(W)$ is the $R$-algebra generated by $u_{I, J}$, $d_{J, I}$ and $1_K$ subject to the  relations ($\tilde {\rm J}1$)--($\tilde {\rm J}3$) obtained by
replacing $\tilde\upsilon_{I,J},\tilde\delta_{J,I},\tilde e_K$ with $u_{I, J}$, $d_{J, I}$ and $1_K$, respectively, and the {\sf quiver relations}
\begin{itemize}
\item[($\tilde{\rm J}$4)] $u_{I, J}1_J=u_{I, J}=1_Iu_{I, J}$, $d_{J, I}1_I=d_{J, I}=1_Jd_{J, I}$ and  $1_I^2=1_I$, $1_I1_J=0$ for $I\not= J$.
\end{itemize} 
\end{corollary}
\begin{proof}
It follows from  Propositions \ref{pres1} and \ref{pres2} and Theorem \ref{tildeQ}.
\end{proof}

\begin{remark} \label{d'}
Let $\sZ=\mathbb Z[q^{\frac12},q^{-\frac12}]$ and let $\sZ E_q$ be the Hecke endomorphism  algebra obtained by base change to $\sZ$.
The generators in Theorem \ref{pres3} are scaled version of a subset of the generators for $\sZ E_q$ given in \cite{GW} by Williamson. 
More precisely,  $u_{I, J}, \frac{\pi(J)}{\pi(I)}d_{J, I}$ with $I\sqsubset J$ form a subset of Williamson's generators.
Let $$d'_{J, I}=  \frac{\pi(J)}{\pi(I)}d_{J, I}.$$
As each $\pi(I)$ is invertible in $R$, if we replace 
the relations ($\tilde{\rm J}$1) and ($\tilde{\rm J}$3) in the theorem by the following, we obtain a presentation of $RE_q$ in terms of the subset of Williamson's
 generators.
\begin{itemize}
\item[($\tilde{\rm J}$1$'$)] $d'_{J, I}u_{I, J}= \frac{\pi(J)}{\pi(I)}1_{J}$.
\item[($\tilde{\rm J}$3$'$)] $u_{\emptyset, I_{1}}\dots u_{I_{m-1}, I}  d'_{I, I_{m-1}}\dots d'_{I_2, I_1}d'_{I_1, \emptyset}=
x_{I}|_{\{t_s\}_{s\in S}}$ for any $I\in\La$ and sequence
$$\emptyset=I_0\sqsubset I_1\sqsubset I_2\sqsubset \dots \sqsubset I_m=I.$$
\end{itemize} 
Williamson's other generators  are monomials in  $d'_{J, I}$ and $u_{I, J}$ with $I\sqsubset J$.  
We will discuss the presentation of $E_q$ over $\mathbb{Z}[q]$ in the last two sections.
\end{remark}

\begin{remark}
Applying the functor $-\otimes_{R}\mathbb{Q}(q)$, 
Theorem \ref{pres3} gives a presentation of the Hecke endomorphism  algebra over $\mathbb{Q}(q)$, i.e. 
$$\mathbb{Q}(q)E_q=\mathbb{Q}(q)\otimes_{R}RE_q=\mathbb{Q}(q)\otimes_{\mathbb{Z}[q]}E_q.$$

Similarly, applying the functor $\mathbb{Z}\otimes_{R}-$, where $\mathbb Z$ is an $R$-module via the ring homomorphism $R\to\mathbb{Z}, q\mapsto 0$, 
Theorem \ref{pres3} gives a presentation of the $0$-Hecke endomorphism algebra 
$E_0=\mathbb{Z}\otimes_{\mathbb{Z}[q]}E_q.$
We will  discuss $0$-Hecke endomorphism algebras in the next section and show that the relations can be simplified.
\end{remark}

\section{0-Hecke endomorphism algebras}
The first application of Theorem \ref{pres3} is to give a presentation of the 0-Hecke endomorphism algebra $E_0$.
Recall the Hecke algebra $H_q=H_q(W)$ over $\mathbb Z[q]$ and its canonical basis $\{C^+_w\}_{w\in W}$ defined in $\eqref{C+}$. 

Consider the specialisation $\mathbb{Z}[q]\to\mathbb Z$ with $q\mapsto 0$ and let $H_0=\bZ\otimes_{\mathbb{Z}[q]}H_q$. We may regard $H_0$ as a degenerate Hecke algebra.
Let 
$$\scc_w=C^+_w\otimes 1, 
\mbox{ i.e. } \scc_w=C^+_w|_{q=0}.$$
Recall that $H_0$ is a $\mathbb{Z}$-algebra generated by 
$T_1, \dots, T_n$ subject to the relations $T_i^2=-T_i$ and the braid relations. 
We call these relations \emph{$0$-Hecke relations}. Note that in the case of the longest element $w_I$,
 $$\scc_{w_I}=\sum_{w\in W_I}T_w=x_I.$$
In particular, $\scc_{s}=T_s+1=x_{\{s\}}$ for all $s\in S$.

Let $w_{I,J}$ be the longest element in the double coset $W_IW_J$. Denote $\scc_{w_{I, J}}$ by $\scc_{I, J}$. Then $\scc_{I,I}=\scc_{w_I}$. Let $(W,*)$ be the Hecke monoid defined in \cite[\S4.5]{DDPW}.

\begin{lemma}\label{canonical}
\begin{itemize}\item[(1)]
The elements $\scc_w,w\in W$, under the multiplication in $H_0$ form a monoid, which is isomorphic to the Hecke monoid $(W,*)$.  In particular, they form a multiplicative basis for $H_0$.
\item[(2)] $\scc_I$  is an idempotent for every finitary $I\in\La^*$.
\item[(3)]
 For any finitary subsets $I,J\subseteq S$, $\scc_{I,I}\scc_{J,J}=\scc_{I,J}.$
 \end{itemize}
\end{lemma}
\begin{proof}
By \eqref{CsCw}, we have 
\begin{equation} \label{0canonicalmult}
\scc_s\scc_w=\begin{cases} \scc_w,&\text{ if }sw<w;\\ \scc_{sw},&\text{ if }sw>w.\\\end{cases}
\end{equation}
Thus, $\{\scc_s \mid s\in S\}$ generate a monoid whose associated monoid algebra is $H_0$. By \cite[Prop.~4.34]{DDPW}, the map $\scc:(W,*)\to H_0,w\mapsto \scc_w$ is the required monoid isomorphism, proving (1). Statements (2) and (3) are consequences of \eqref{0canonicalmult}.
\end{proof}

We call $\{\scc_w\mid w\in W\}$  the {\it degenerate canonical basis} of $H_0$. The following is clear from \eqref{0canonicalmult} and Lemma \ref{canonical} (1).

\begin{corollary} \label{xy}
Let $I,J,K$ be finitary subsets of $S$ and let $x\in D_{I,J}^+$ and $y\in D_{J,K}^+$.
 Then $\scc_x\scc_y=\scc_{x*y}$. In particular, $x*y\in D_{I,K}^+$.
\end{corollary}

\begin{lemma}\label{newrel}  
The $0$-Hecke algebra $H_0$ is generated by $\scc_s, s\in S$, subject to the relations.
\begin{itemize}
\item[(1)] 
$\underbrace{\scc_s\scc_t\dots}_{m_{s,t}}=\underbrace{\scc_t\scc_s\dots}_{m_{s,t}}$, where $m_{s,t}$ is the order of $st$ in $W$.
\item[(2)]  $\scc_s^2=\scc_s$ for all $s\in S$.
\end{itemize}
\end{lemma}

\begin{proof}
The equivalence between the relations in (1)
 and the braid relations follows from (\ref{0canonicalmult}) and clearly    $T_i^2=-T_i$ implies that $(T_i+1)^2=(T_i+1)$ and 
vice versa. So the claim holds.
\end{proof}

We now give an interpretation of the 0-Hecke endomorphism algebra in terms of the degenerate canonical basis $\scc_w,w\in W$.
Let  
$$\sB=\{(I,\scc_{d},J)\mid I,J\in\La,d\in D_{I,J}^+\}.$$
Form the free $\mathbb Z$-module $\mathbb Z\sB$ with basis $\sB$ and define multiplication with
$$(I,\scc_x,J)*(K,\scc_y,L)=\delta_J^K(I,\scc_x\scc_y,L),$$
where $\delta_J^K=1$ if $J=K$ and 0 otherwise.
Note that, by Corollary~\ref{xy}, the multiplication is well-defined. Note also that $f_I=(I,\scc_{I,I},I)$ is an idempotent and $f_\emptyset \sB f_\emptyset$ is isomorphic to the Hecke monoid $(W,*)$.

\begin{theorem} With this multiplication, $\mathbb Z\sB$ becomes an associative algebra with the identity $1=\sum_{I\in\La}f_I$. Moreover, we have $E_0\cong \mathbb Z\sB$.
\end{theorem}

\begin{proof} The isomorphism can be easily seen by using the degenerate canonical basis 
$$\{\Theta_{I,J}^d\mid I,J\in\La,d\in D_{I,J}^+\}$$ for $E_0$, where $\Theta_{I,J}^d\in E_0$ is defined by
(cf. \cite[Th.~2.3]{Du}) 
$$\Theta_{I,J}^d(\scc_{w_K}h)=\delta_J^K\scc_dh$$
for all $K\in\Lambda,h\in H_0$. By \eqref{0canonicalmult}, it is easy to see that $\Theta_{I,J}^x\Theta_{J,K}^y(\scc_{w_K})=\scc_x\scc_y=\scc_{x*y}=\Theta_{I,K}^{x*y}(\scc_{w_K}).$ Hence, $\Theta_{I,J}^x\Theta_{J,K}^y=\Theta_{I,K}^{x*y}$ and the required isomorphism is the map $(I,\scc_x,J)\mapsto \Theta_{I,J}^x$.
\end{proof}

Following Corollary \ref{pres3} and applying the functor $\mathbb{Z}\otimes_{R}-$, we have 
the following presentation of the $0$-Hecke endomorphism algebras $E_0$.

\begin{theorem}\label{pres4}The algebra $E_0$  with $1=\sum_{I\in \Lambda}1_I$ is generated by $u_{I, J}$, $d_{J, I}$ and $1_K$, where $I,J,K\in\La$ with
$I\sqsubset J$ subject to the  relations 
\begin{itemize}
\item[(1)] $d_{J, I}u_{I, J}=1_{J}$ for all $I,J\in\La$ with
$I\sqsubset J$ ;
\item[(2)] $u_{I, J}u_{J, K}=u_{I, J'}u_{J', K}$ and $d_{K, J}d_{J, I}=d_{K, J'}d_{J', I}$ for all $I,J,J',K\in\Lambda$ with $I\sqsubset J\sqsubset K$ and $I\sqsubset J'\sqsubset K$;
\item[(3)] $u_{\emptyset, I_{1}}\dots u_{I_{m-1}, I}d_{I, I_{m-1}}\dots d_{I_2, I_1}d_{I_1, \emptyset}=p_{\{s_l\}}\cdots p_{\{s_2\}}p_{\{s_1\}}$,
for any $I\in \La^*$, and any reduced expression of $w_{I_\bullet}=s_i\cdots s_2s_1$ associated with the ordering
$$I_\bullet:\emptyset=I_0\sqsubset I_1\sqsubset I_2\sqsubset \dots \sqsubset I_m=I;$$ 
\item[(4)] $u_{I, J}1_J=u_{I, J}=1_Iu_{I, J}$, $d_{J, I}1_I=d_{J, I}=1_Jd_{J, I}$, $1_I^2=1_I$ and $1_I1_J=0$ ($I\not= J$).
\end{itemize} 
\end{theorem}

The generators $u_{I, J}$, $d_{J, I}$ and $1_K$ correspond to elements $(I,\scc_{I,J},J), (J,\scc_{J,I},I)$ and $f_K$ in $\sB$. Theorem \ref{pres4} gives a presentation for the algebra $\mathbb Z\sB$ by the generators
$(I,\scc_{I,J},J), (J,\scc_{J,I},I)$ for all $I,J\in\La$ with $I\sqsubseteq J$. 

We end this section with an algorithm for writing a basis element in $\sB$ as product of these generator.
For a $W_I$-$W_J$ double coset $\fp$, let $\fp^+,\fp^-$ be the longest, shortest elements in $\fp$. For $w\in W$, let 
$$\sR(w)=\{s\in S\mid ws<w\}.$$ 
The following algorithm gives a way to write a basis element $(I,\scc_{\fp^+},J)$ as a product of generators. This algorithm is a degenerate version of the one given in the proof of \cite[Prop.~1.3.4]{GW1}. Recall the notation used in \eqref{ladmu}.

Given a double coset $(I,\fp,J)$ with $I,J\in\La$,
 let $\tilde J=\sR(\fp^+)$, $J_1=I^{\fp^-}\cap \tilde J$ and $\fp_1=W_I\fp^-W_{J_1}$. Then 
 $$\scc_{\fp^+}=\scc_{\fp_1^+}\scc_{w_{\tilde J}}=\scc_{\fp_1^+}\scc_{J_1,\tilde J}\scc_{\tilde J,J}.$$
 If $J_1=\tilde J$, then $I=\tilde J$ and $\fp^-=1$. So the algorithm stops. Otherwise,
continuing the algorithm with $(I,\fp,J)=(I,\fp_1,J_1)$ yields $\tilde J_1\in\La$ and double coset $(I,\fp_2,J_2)$ such that $J_2\subset\tilde J_1\supset J_1$ and  
$$\scc_{\fp^+}=\scc_{\fp_2^+}\scc_{J_2,\tilde J_1}\scc_{\tilde J_1,J_1}\scc_{J_1,\tilde J}\scc_{\tilde J,J}.$$
The algorithm stops if $J_2=\tilde J_1$.
In this way, we eventually find sequences $J_0=J,J_1,\ldots, J_m=I$ and $\tilde J_0,\tilde J_1,\ldots, \tilde J_{m-1}=I$ in $\La$ such that $J_i\subset \tilde J_{i-1}\supset J_{i-1}$ ($1\leq i\leq m-1$) and 
$$\scc_{\fp^+}=\scc_{I, \tilde J_{m-1}}\scc_{\tilde J_{m-1},J_{m-1}}\cdots \scc_{J_2,\tilde J_1}\scc_{\tilde J_1,J_1}\scc_{J_1,\tilde J}\scc_{\tilde J,J}.$$
Thus, we write
$(I,\scc_{\fp^+},J)$ as a product of generators.

\section{The integral case}\label{rank2}

It seems difficult to find the generating relations for the integral algebra $E_q(W)$ over $\mathbb{Z}[q]$. In this section,  we first explain
how to find torsion relations. 
We then compute the integral  presentation in  the rank 2 case (see Theorem \ref{rank2relations}). A more complicated example  in the $A_3$ case is also computed  in the last section.

For notational simplicity, we will denote the Hasse quiver $\tilde Q$ in Section 5 by $\sQ$ and will drop the $\tilde{\ }$'s on all its arrows and trivial paths.\footnote{The reader should not be confused with the quiver $Q$ discussed in Section 4.} Thus, the vertex set of $\sQ$ is $\La=\Lambda(W)$ and the arrow set is
$\{\delta_{J,I},\upsilon_{I,J}\mid I,J\in\Lambda,I\sqsubset J\}$ with trivial paths $e_K$, $K\in\La$.

Let $\sQ_q=\sQ_q(W):=\mathbb Z[q]\sQ$ be the path algebra of $\sQ$ over $\mathbb Z[q]$.  For $s\in S$, let 
$$\chi_{s}= \upsilon_{\emptyset, \{s\}}\delta_{\{s\}, \emptyset}.$$
Note that the factor $(q+1)$ in $\tilde\chi_s$ is dropped here.

For any $I\in\La$ and given ordering $I_\bullet:\emptyset=I_0\sqsubset I_1\sqsubset\dots \sqsubset I_m=I,$ use $\chi_s$ to define $\chi_{I_\bullet}$ similarly as defining $\tilde\chi_{I_{\bullet}}$ in \eqref{chi_I}.


Let $\sJ$ be the subset of $\sQ_q$ defined by the following relations
\begin{itemize}
\item[$({\mathcal{J}1})$] {\sf Quasi-idempotent relations:} ${\delta}_{J, I}{\upsilon}_{J, I}=\frac{\pi(J)}{\pi(I)}{e}_J.$ 
\item[$({\mathcal{J}2})$] {\sf Sandwich relations:} $$\upsilon_{I, J} \upsilon_{J, K}= \upsilon_{I, J'} \upsilon_{J', K}\text{ and } \delta_{K, J} \delta_{J, I}= \delta_{K, J'} \delta_{J', I}$$ for all $I,J,J',K\in\Lambda$ with
$I\sqsubset J\sqsubset K$ and $I\sqsubset J'\sqsubset K.$
\item[$({\mathcal{J}3})$]  {\sf Extended braid relations}: \vspace{-1ex}
$$ \upsilon_{\emptyset, I_{1}}\dots  \upsilon_{I_{m-1}, I} \delta_{I, I_{m-1}}\dots  \delta_{I_2, I_1} \delta_{I_1, \emptyset}= \chi_{I_\bullet}$$
 for all $I\in \La^*$ and given ordering $I_\bullet$.
\end{itemize}

If we put $\tilde{\delta}_{J, I}=\frac{\pi(J)}{\pi(I)}{\delta}_{J, I}$, then relations $({\mathcal{J}1})$ and $({\mathcal{J}3})$ give $({\tilde{\rm J}1})$ and $({\tilde{\rm J}3})$. Hence, in $R\sQ_q$, $\mathcal J$ generates the ideal $\langle\tilde{\mathcal J}\rangle$.

As in the proof of Corollary \ref{pres3}, we have a surjective algebra homomorphism
$$
\psi: {\sQ_q}/\langle {\mathcal{J}}\rangle \to E_q(W)
$$
and thus a short exact sequence
\begin{equation}\label{last2}
0\longrightarrow \mathcal{K}/\langle {\mathcal{J}}\rangle \longrightarrow \sQ_q/\langle {\mathcal{J}}\rangle \longrightarrow E_q(W)\longrightarrow 0,
\end{equation}
for some ideal $\mathcal K$ of $\sQ_q$.
Applying $R\otimes _{\mathbb{Z}[q]}-$  gives the short exact sequence
\begin{equation}
0\longrightarrow R\otimes _{\mathbb{Z}[q]}  \mathcal{K}/\langle {\mathcal{J}}\rangle \longrightarrow  R\otimes _{\mathbb{Z}[q]} \big(\sQ_q/\langle {\mathcal{J}}\rangle\big) \longrightarrow
R\otimes _{\mathbb{Z}[q]} E_q(W)\longrightarrow 0.
\end{equation}
Since $R\otimes _{\mathbb{Z}[q]} \big(\sQ_q/\langle {\mathcal{J}}\rangle\big)\cong R\tilde Q/\langle {\tilde{\mathcal{J}}}\rangle$,
by Corollary \ref{pres3}, $$ R\otimes _{\mathbb{Z}[q]}  \mathcal{K}/\langle {\mathcal{J}}\rangle=0.$$ 
Hence, to give a presentation of $E_q(W)$, we need to find sufficient 
\emph{$P$-torsion elements}, 
which are elements in  $\mathcal{K}/\langle {\mathcal{J}}\rangle$ annihilated by a polynomial in $P$. Note that $P$ is the multiplicative 
set of polynomials in $\mathbb{Z}[q]$ with the constant term $1$.
Below we explain how to find those elements.

For any $I\in\La$ and any given sequence $\emptyset=I_0\sqsubset I_1\sqsubset\dots \sqsubset I_m=I,$ let
$$\aligned
\upsilon_I&=\upsilon_{\emptyset,I}= \upsilon_{\emptyset, I_1}\upsilon_{I_1,I_2}\dots  \upsilon_{I_{m-1}, I},\\
\delta_I&=\delta_{I, \emptyset}= \delta_{I, I_{m-1}}\dots  \delta_{I_2, I_1} \delta_{I_1, \emptyset}.
\endaligned
$$
In general, for $I\subset J$ in $\Lambda$, let $\delta_{J, I}$ (resp., $\upsilon_{I, J}$) denote the shortest paths from $I$ (resp. $J$) to $J$ (resp. $I$). 

\begin{proposition}\label{general1}
Let $p=\sum_i f_ip_i$ be a linear combination of paths from $I$ to $J$, where each $f_i\in \mathbb{Z}[q]$. If $\upsilon_{\emptyset, J} p \delta_{I, \emptyset}=0$ in $\sQ_q/\langle {\mathcal{J}}\rangle$, then $p$ is $P$-torsion. 
Further, the ideal $ \mathcal{K}$ is generated by all such $P$-torsion elements and $\sJ$. 
\end{proposition}

\begin{proof}
As $\delta_I\upsilon_I=\pi(I){e}_I$ for any $I\in \Lambda$, then in $\sQ_q/\langle {\mathcal{J}}\rangle$
$$
\pi(J)\pi(I)p= \delta_{J}\upsilon_{J} p \delta_{I}\upsilon_{I}=0.
$$
So $p$ is $P$-torsion.

Now suppose that $p=\sum_i f_ip_i$ is $P$-torsion, we may assume that all the paths $p_i$ start from $I$ and end at $J$. Then there is a  polynomial $f\in R$ such that $fp=0$ and 
$$
\upsilon_J f p \delta_I=0 \in \bar e_\emptyset \big(\sQ_q/\langle {\mathcal{J}}\rangle\big) \bar e_\emptyset.
$$
Note that (cf. the proof of Theorem \ref{tildeQ})
$$
\bar e_\emptyset \big(\sQ_q/\langle {\mathcal{J}}\rangle\big) \bar e_\emptyset \cong H_q,
$$
which is a free $\mathbb{Z}[q]$-module. 
Therefore $\upsilon_J  p \delta_I=0$. This says that $p$ is $P$-torsion as described and so 
the ideal $\mathcal{K}$ is generated by all such $P$-torsion elements and $\sJ$.
\end{proof}

The $P$-torsion elements in $\sQ_q/\langle {\mathcal{J}}\rangle$ give us new relations. We call them {\it torsion relations}. 

%
%

\begin{lemma}\label{pathsviathebottom}
If $(W, S)$ is  finite, then any path from $I$ to $S$ (resp. from $S$ to $I$) is a multiple of  $\delta_{I,S}$ (resp. $\upsilon_{S,I}$).
\end{lemma}
\begin{proof}
The proof follows from the sandwich relations and the quasi-idempotent relations.
\end{proof}

We now look for the torsion relations for dihedral Hecke algebras $\mathbb Z[q]$-algebra $E_q(I_n)$. 
In this case, Elias gave an recursive presentation of the Hecke endomorphism algebras in \cite[Prop.~2.20]{Elias} over $\mathbb{Z}[v, v^{-1}]$, where $v^2=q$. 
We work over $\mathbb{Z}[q]$ and can write down explicit generating relations.  
Our method is different  from Elias'. 

Let $S=\{1,2\}$. Note that the Hasse quiver $\mathcal{Q}$ has the form: 
$$\mathcal{Q}:
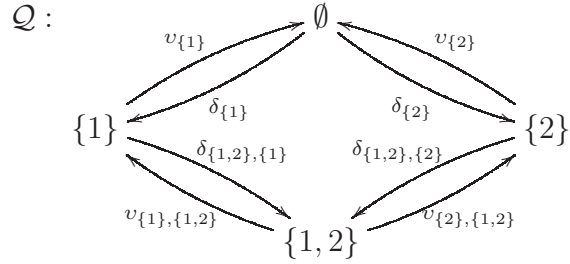
\xymatrix{
&& \emptyset  \ar@/^/[dll]^{\delta_{\{1\}}} \ar@/_/[drr]_{\delta_{\{2\}}}&\\
\{1\} \ar@/^/[drr]^{\delta_{\{1,2\}, \{1\}}}  \ar@/^/[urr]^{\upsilon_{\{1\}}} &&&&
\{2\} \ar@/_/[dll]_{\delta_{\{1,2\}, \{2\}}}  \ar@/_/[ull]_{\upsilon_{\{2\}}} \\&
&\{1, 2\} \ar@/_/[urr]_{\upsilon_{\{2\}, \{1, 2\}}}\ar@/^/[ull]^{\upsilon_{\{1\}, \{1, 2\}}} &
}
$$

Let $\chi_{1[j]}$ and $\chi_{2[j]}$ be monomials in $\chi_{i}$, defined in a similar way as $x_{1[j]}$ and $x_{2[j]}$.
Let
$\chi_{(1, 2)}^{(m)}$ be the  polynomial in the new $\chi_{1}$ and $\chi_{2}$, as given in Lemma \ref{lastlemma}, 
and $\chi_{(2, 1)}^{(m)}$ the polynomial obtained from $\chi_{(1, 2 )}^{(m)}$ by swapping the indices $1$ and $2$. Thus, \eqref{equation} and its counterpart for (2,1) can be combined into the following: for $(s,t)=(2,1)$ or $(1,2)$,
$$
\chi^{(m)}_{(s,t)}=\sum_{j=1}^{m}b_j^m\chi_{[j]t}=\sum_{i=0}^{\lfloor \frac{m-1}{2}\rfloor}{m-i-1\choose i}(-q)^{i}\chi_{[m-2i]t}.$$

\begin{theorem}\label{rank2relations} 
The $\mathbb Z[q]$-algebra $E_q(I_n)$ is isomorphic to the quotient algebra 
of the path algebra $\mathcal{Q}_q$ modulo the following relations $\sJ_n$.
\begin{itemize}
\item[(1)] {\sf Quasi-idempotent relations:} for $i=1,2$,
\begin{itemize}
\item[(i)] $\delta_{\{i\}}\upsilon_{\{i\}}=(1+q)e_{\{i\}}$;
\item[(ii)] $\delta_{\{1, 2\}, {\{i\}}}\upsilon_{{\{i\}}, \{1, 2\}}=(1+q+\cdots+q^{n-1})e_{\{1, 2\}}$.
\end{itemize}
\item[(2)] {\sf Sandwich relations: }
\begin{itemize}
\item[(i)]$\upsilon_{\{1\}}\upsilon_{\{1\}, \{1,2\}}=\upsilon_{\{2\}}\upsilon_{\{2\}, \{1,2\}}$ 
\item[(ii)]$\delta_{\{1,2\}, \{1\}}\delta_{\{1\}}=\delta_{\{1,2\}, \{2\}}\delta_{\{2\}}$;
\end{itemize}
\item[(3)]  {\sf Refined braid relations:} for $(s, t)=(1, 2)$ or $(2, 1)$. 
\begin{itemize}
\item[(i)] $\upsilon_{\{t\}, \{1,2\}}\delta_{\{1,2\}, \{s\}}=\delta_{\{t\}} \big(\sum_{j=2}^n b_j^n\chi_{[j-2]t}\big)\upsilon_{\{s\}}$, where $\chi_{[0]t}=e_\emptyset$, if $n$ is even;
\item[(ii)] $\upsilon_{\{s\}, \{1,2\}}\delta_{\{1,2\}, \{s\}}=\delta_{\{s\}} \big(\sum_{j=3}^n b_j^n\chi_{[j-2]t}\big)\upsilon_{\{s\}}+(-q)^{\frac{n-1}{2}}e_{\{s\}}$, if $n$ is odd.
\end{itemize}
\end{itemize}\end{theorem}

\begin{remark}\label{pre-braid}
Note that the braid relation $\chi_{(1, 2 )}^{(m)}=\chi_{(2,1 )}^{(m)}$ in Lemma \ref{NewpreHecke} can be easily derived by multiplying the refined braid relations with $\upsilon_{\{t\}}$ (or $\upsilon_{\{s\}}$) on the left and $\delta _{\{s\}}$ on the right and applying the sandwich relations. 
\end{remark}

\begin{lemma}\label{torsion} The following relations hold in $\mathcal{Q}_q/\langle \mathcal{J}_n\rangle$: for $(s, t)=(1, 2)$ or $(2, 1)$. 
\begin{itemize}
\item[(a)] $(q+1)\upsilon_{\{s\}, \{1,2\}}\delta_{\{1,2\}, \{s\}}=\delta_{\{s\}} (\sum_{j=1}^n b_j^n\chi_{[j-1]t})\upsilon_{\{s\}}$, if $n$ is even;
\item[(b)] $(q+1)\upsilon_{\{t\}, \{1,2\}}\delta_{\{1,2\}, \{s\}}=\delta_{\{t\}} (\sum_{j=1}^n b_j^n\chi_{[j-1]t})\upsilon_{\{s\}}$,  if $n$ is odd,
\end{itemize}
where $\chi_{[0]t}=e_\emptyset$.
\end{lemma}
\begin{proof} We only prove (a). By the quasi-idempotent relation and the sandwich relation, the left hand side 
$$\text{LHS}=(\delta_{\{s\}}\upsilon_{\{s\}})\upsilon_{\{s\}, \{1,2\}}\delta_{\{1,2\}, \{s\}}=\delta_{\{s\}}\upsilon_{\{t\}}(\upsilon_{\{t\}, \{1,2\}}\delta_{\{1,2\}, \{s\}}).$$
Now, applying (3)(i) yields (a).
\end{proof}


\begin{proof}[Proof of Theorem \ref{rank2relations}]
Let $\phi:\mathcal{Q}_q/\langle \mathcal{J}_n\rangle \to E_q(I_n)$ be the algebra homomorphism given by sending, for all $I,J,K\in\Lambda$ with $I\sqsubset J$, $e_K$ to the identity map $1_K$ on $x_KH_q$,
$\delta_{J,I}$ to the map $d_{J,I}'=\frac{\pi(J)}{\pi(I)}d_{J, I}:x_IH_q\to x_JH_q$ (see Remark \ref{d'}) and $\upsilon_{I,J}$ to the inclusion map $u_{I,J}:x_JH_q\to x_IH_q$. Then  $x_{\{i\}}=u_{\emptyset,\{i\}}d_{\{i\},\emptyset}': H_q\rightarrow H_q$ is given by  
right multiplication with $x_{\{i\}}$.
Direct  computation shows that  all three sets of relations are satisfied if $e_K$, $\delta_{J,I}$ and $\upsilon_{I,J}$ are replaced by $1_K, d'_{J, I}$ and $u_{I, J}$. For example, both sides of the refined braid relations become the maps sending $x_{\{s\}}$ to $x_{\{1,2\}}$. 
So the map $\phi$ is indeed a well-defined algebra homomorphism.
Note that $1_K, u_{I, J} $ and $d'_{J, I}$ for all $K, I, J\in \Lambda$ with $I\subset J$  generate the $\mathbb{Z}[q]$- algebra 
$E_q(I_n)$ \cite[Prop.~2.11]{GW}. Further, for any $I\subset J$, $u_{I, J}$ and $d'_{J, I}$ are  
monomials in $u_{K, L}$ and $d'_{L, K}$ with $K\sqsubset L$. So $\phi$ is surjective.

We show now that ${\phi}$ is injective for {\it even} $n$, the odd case can be done similarly. 
Since $\bar e_\emptyset \big(\mathcal{Q}_q/\langle \mathcal{J}_n\rangle\big )\bar e_{\emptyset}\cong H_q$, it suffices by Lemma \ref{pathsviathebottom} and symmetry, to prove that the following sets $B_{J,I}$ span the space of  paths from $I$ to $J$ in $\mathcal{Q}_q/\langle \mathcal{J}_n\rangle$,
$$\aligned
&B_{\{1\},\emptyset}= \{\delta_{\{1\}} \chi_{2[j]}\mid 0\leq j\leq n-1\};\\
&B_{\{1\},\{1\}}=\{ e_{\{1\}},  \upsilon_{\{1\}, \{1, 2\}}\delta_{\{1, 2\}, \{1\}}\}\cup \{ \delta_{\{1\}}\chi_{[2j-1]2}\upsilon_{\{1\}}\mid 1\leq j\leq \frac n2-1\};\\
&B_{\{2\},\{1\}}=\{ \delta_{\{2\}}\upsilon_{\{1\}},  \upsilon_{\{2\}, \{1, 2\}}\delta_{\{1, 2\}, \{1\}}\}\cup \{ \delta_{\{2\}}\chi_{[2j-2]2}\upsilon_{\{1\}}\mid 2\leq j\leq \frac n2-1\}.
\endaligned$$
Note that the cardinality of $B_{J,I}$ is the same as the cardinality of the double coset $W_J\backslash I_n/ W_I$, i.e. the rank of the homomorphism space from $x_IH_q$ to $x_JH_q$, 
as a free $\mathbb{Z}[q]$-module.  So
we have the injectivity. 

We now prove the claim for the set $B_{\{1\},\{1\}}$, the others can be done similarly. By Lemma \ref{pathsviathebottom}, any path from $\{1\}$ to $\{1\}$ via $\{1, 2\}$ 
is a multiple of $\upsilon_{\{1\}, \{1, 2\}}\delta_{\{1, 2\}, \{1\}}.$  So any path from  $\{1\}$ to $\{1\}$ 
is a multiple of a path 
$ e_{\{1\}},  \upsilon_{\{1\}, \{1, 2\}}\delta_{\{1, 2\}, \{1\}}$ or $\delta_{\{1\}}\chi_{[j]s}\upsilon_{\{1\}}$, where $j\geq 1$ and $s=1$ or $2$. 
When $s=1$, or  $j$ is even and $s=2$, by the quasi-idempotent relations in (1), $\delta_{\{1\}}\chi_{[j]s}\upsilon_{\{1\}}$ is equal to a multiple of  $\delta_{\{1\}}\chi_{[l]2}\upsilon_{\{1\}}$ with $l<j$.  
So we need only to prove that 
$\delta_{\{1\}}\chi_{[l]2}\upsilon_{\{1\}}$, for {\it odd} numbers $l\geq n-1$, is a linear combination of the paths in $B_{\{1\},\{1\}}$. Indeed, by the relation in Lemma \ref{torsion}(a),
$$\aligned
&\delta_{\{1\}}\chi_{[n-1]2}\upsilon_{\{1\}}
=(q+1)\upsilon_{\{1\}, \{1,2\}}\delta_{\{1,2\}, \{1\}} - \delta_{\{1\}} \bigg(\sum_{j=1}^{n-2} b_j^n\chi_{[j-1]2}\bigg)\upsilon_{\{1\}} \in \text{ span}(B_{\{1\},\{1\}})\\
\endaligned
$$
and, for odd $l\geq n+1$,
$$\aligned
\delta_{\{1\}}\chi_{[l]2}&\upsilon_{\{1\}}=\big(\delta_{\{1\}}\chi_{[n-1]2}\upsilon_{\{1\}}\big)\delta_{\{1\}}  \chi_{[l-n]2}\upsilon_{\{1\}}\\
&=(q+1)\upsilon_{\{1\}, \{1,2\}}\delta_{\{1,2\}, \{1\}} \delta_{\{1\}}    \chi_{[l-n]2} \upsilon_{\{1\}}- \delta_{\{1\}} \bigg(\sum_{j=1}^{n-2} b_j^n\chi_{[j-1]2}\bigg)\upsilon_{\{1\}}\delta_{\{1\}}    \chi_{[l-n]2} \upsilon_{\{1\}}\\
&=(q+1)\upsilon_{\{1\}, \{1,2\}}\delta_{\{1,2\}, \{2\}} \delta_{\{2\}} \chi_{\{2\}} \chi_{[l-n-1]2}  \upsilon_{\{1\}}- \delta_{\{1\}} \bigg(\sum_{j=1}^{n-2} b_j^n\chi_{[l+j-n]2}\bigg)\upsilon_{\{1\}}.\\
\endaligned
$$
Further, repeatedly applying $\delta_{\{s\}}\chi_{\{s\}}=(q+1)\delta_{\{s\}}$ and  the downward sandwich relation on 
$\upsilon_{\{1\}, \{1,2\}}\delta_{\{1,2\}, \{2\}} \delta_{\{2\}} \chi_{\{2\}} \chi_{[l-n-1]2}  \upsilon_{\{1\}}$
leads to
$$
\aligned
\delta_{\{1\}}\chi_{[l]2}\upsilon_{\{1\}}   
=(q+1)^{l-n+2}\upsilon_{\{1\}, \{1,2\}}\delta_{\{1,2\}, \{1\}}-\delta_{\{1\}} \bigg(\sum_{j=1}^{n-2} b_j^n\chi_{[l+j-n]2}\bigg)\upsilon_{\{1\}}.\\
\endaligned
$$
So, by induction, $\delta_{\{1\}}\chi_{[l]2}\upsilon_{\{1\}}$ is a linear combination of the paths in $B_{\{1\},\{1\}}$, as required. 
Therefore, we can conclude that the two algebras are isomorphic.
\end{proof}

\section{The integral case of type $A_3$}\label{generalcase}

In this section, we will compute the torsion relations required for presenting $E_q(\fS_4)$.
Recall in this case $S=\{1, 2,3\}$, 
the Hasse quiver  $\sQ$ given in \eqref{HasseA3}, and the path algebra $\sQ_q=\mathbb Z[q]\sQ$. 
Note that the Dynkin graph automorphism that interchanges 1 and 3 and fixes 2 induces a {\it graph automorphism} 
$\sigma:\sQ_q\longrightarrow\sQ_q.$
We will also use {\it subgraph automorphisms} induced by the graph automorphism on parabolic subgroups.

Given a path $p$ from $I$ to $J$, denote the path $\upsilon_{\emptyset, I}p\delta_{J,\emptyset}$ by $p^\wedge$. 

\begin{lemma}\label{Rtorsion}
Maintain the notation introduced in Section \ref{rank2}.
The following are torsion relations in $\sQ_q/\langle\mathcal J\rangle$.\footnote{Here each relation should be understood as the set of the relations obtained by applying (sub)graph 
automorphisms and the anti-involution \eqref{tau} to the relation. Thus, (T1) really represents 4 relations.}
$$\aligned
(\text{\rm T}1) \;& \delta_{\{1\}, \emptyset}\chi_2 \upsilon_{\emptyset, \{1\}}=  \upsilon_{\{1\}\{1, 2\}} \,\delta_{\{1, 2\}, \{1\}}+qe_{\{1\}}; \\
(\text{\rm T}2)\; & \delta_{\{1\}, \emptyset} \upsilon_{\emptyset, \{3\}}=\upsilon_{\{1\}, \{1, 3\}} \,\delta_{\{1, 3\}, \{3\}};\\
%
(\text{\rm T}3) \;& \delta_{\{1, 2\}, \{1\}} \upsilon_{\{1\}, \{1, 3\}} \delta_{\{1, 3\}, \{1\}} \upsilon_{\{1\}, \{1, 2\}}=\upsilon_{\{1, 2\}, S}\delta_{S, \{1, 2\}}+
q(q+1) e_{\{1, 2\}};\\
(\text{\rm T}4) \;& \delta_{\{1, 2\}, \{2\}}\upsilon_{\{2\}, \{2, 3\}}\delta_{\{2, 3\}, \{3\}}\upsilon_{\{3\}, \{1, 3\}}= \upsilon_{\{1, 2\}, S}\delta_{S, \{1, 3\}}+ q\delta_{\{1, 2\}, \{1\}}\upsilon_{\{1\}, \{1, 3\}};\\
\endaligned
$$
\end{lemma}


\begin{proof}
(T1) and (T2) are rank 2 torsion relations from Theorem \ref{rank2relations}. 
We prove (T4), and (T3) can be done similarly. If we put
 $$p=\upsilon_{\{2\},  \{2, 3\}}\delta_{\{2, 3\}, \{3\}}\upsilon_{\{3\}, \{1, 3\}} \delta_{\{1, 3\}, \emptyset} \,\text{ and } \,
l=\upsilon_{\{3\}, \{1, 3\}} \delta_{\{1, 3\}, \emptyset},$$ then 
$$
\aligned
(q+1)({\text{LHS}})^\wedge&=\upsilon_{\emptyset,   \{1, 2\}}\delta_{\{1, 2\}, \{2\}}(q+1)e_{\{2\}} \cdot  p \\
&=(\chi_2\chi_1 \chi_2-q\chi_2)\upsilon_{\emptyset, \{2\}} p,  \quad (\text{by ($\mathcal J$1), ($\mathcal J$3)})\\
&= (\chi_2\chi_1 -qe_\emptyset)\upsilon_{\emptyset, \{2\}} (q+1) p\\
&=(\chi_2\chi_1  - qe_{\emptyset}) \upsilon_{\emptyset, \{2\}}  \upsilon_{\{2\},  \{2, 3\}}\delta_{\{2, 3\}, \{3\}} \cdot (q+1)\cdot l\\
&=(\chi_2\chi_1  - qe_{\emptyset})   (\chi_3\chi_2 \chi_3-q\chi_3)\upsilon_{\emptyset, \{3\}} \cdot l,  \quad (\text{by ($\mathcal J$1), ($\mathcal J$3)})\\
&=(\chi_2\chi_1  - qe_{\emptyset})   (\chi_3\chi_2 -qe_{\emptyset})\upsilon_{\emptyset, \{3\}} (q+1)\cdot l,   \quad (\text{by ($\mathcal J$1)})\\
&=(q+1)( \chi_2\chi_1\chi_3\chi_2-q\chi_2\chi_1-q\chi_3\chi_2+q^2e_{\emptyset}) \upsilon_{\emptyset, \{3\}}  l\\
&=(q+1)( \chi_2\chi_1\chi_3\chi_2-q\chi_2\chi_1-q\chi_3\chi_2+q^2e_{\emptyset})\chi_3\chi_1,  \quad (\text{by ($\mathcal J$3)})\\
&=(q+1)(\chi_S+q\chi_{\{1, 2\}} \chi_3)\\
&=(q+1)[({\delta_{\{1, 2\}, S}\upsilon_{S, \{1, 3\}}})^\wedge+q({\delta_{\{1, 2\}, \{1\}}\upsilon_{\{1\}, \{1, 3\}}})^\wedge], \quad (\text{by ($\mathcal J$3)),($\mathcal J$1)}).\\
\endaligned
$$
So by defnition and Proposition \ref{general1}, (T4) is a torsion relation. 
\end{proof}


Let $\mathcal K$ be the ideal of $\sQ_q$ generated by the quasi-idempotent relations and sandwich relations together the four sets relations (T1)--(T4). By footnote 2,  both the anti-automorphism $\tau$ and the graph automorphism $\sigma$ stabilise $\mathcal K$.
Let
$$A=\sQ_q/ {\mathcal{K}}.$$  
Then $\tau$ and $\sigma$ induces an anti-automorphism and an automorphism on $A$. 
We will use the same letters to denote them.
Observe that  (cf. the proof of Theorem \ref{tildeQ})
\begin{equation}\label{Heckeisom}
\bar e_\emptyset A \bar e_\emptyset \cong H_q.
\end{equation}

Denote by $P(J, I)_\emptyset$ the subspace in $A$, consisting of paths from $I$ to $J$ via $\emptyset$ and 
$\chi_{\underline{w}}=\chi_{s_{i_1}}\dots \chi_{s_{i_t}}$ for a given reduced word $w=s_{i_1}\dots s_{i_t}\in W$.
We call a path of the form  $\upsilon_{L, K} \delta_{K, I}$ {\it a hook}, where $I, L\subset K$. Further, if $|K|=i$, then the hook is called 
a level-$i$ hook. The relations (T1) and (T2) show that a level-2 hook is a linear combination of paths without level-2 hooks.

\begin{lemma}\label{pathsviathetop}
The space $P(J, I)_\emptyset$ is spanned by 
$${\mathscr P}_{J,I}^\emptyset=\{ \delta_{J, \emptyset} \chi_{\underline w} \upsilon_{\emptyset, I}\mid w\in D_{IJ} \text{ and } \underline{w} \text{ is a fixed reduced expression of  } w\}.$$
\end{lemma}
\begin{proof} If both  $I$ and $J$ are  the empty set $\emptyset$, by the isomorphism in (\ref{Heckeisom}),  it is done. Otherwise, by symmetry we may assume that $I\not=\emptyset$.
We claim that any path via the vertex $\emptyset$ is a linear combination of paths $ \delta_{J, \emptyset} Y \upsilon_{\emptyset, I}$, where $Y$ is a path from $\emptyset$ to 
$\emptyset$, and then the lemma again follows from the isomorphism in (\ref{Heckeisom}). Let $p$ be a path in $P(J, I)_\emptyset$. 
We may assume that $p$ is not of the form  $ \delta_{J, \emptyset} Y \upsilon_{\emptyset, I}$. Write 
$p=lr$, where $r$ is minimal ending at $\emptyset$. If $r\in\text{span}({\mathscr P}_{\emptyset,I}^\emptyset)$, by symmetry, $p\in\text{span}({\mathscr P}_{J,I}^\emptyset)$. We need only to prove the claim for $r$.
By the assumption, there is a hook in $r$. 
Write $r=\upsilon_{\emptyset, K} \delta_{K, L} r'$, where $L\subset K$ and $|L|$ is minimal. Further the cardinality $|K|>1$, by assumption.

Now by the sandwich relations (if necessary) together with (T1) or (T2) when $K=2$, or with (T4) when $|K|=3$ (i.e. $K=S$), 
the hook $\upsilon_{\emptyset, K} \delta_{K, L}$
is a linear combination $\sum_{Y}a_YY\upsilon_{\emptyset, L}$, where $a_Y\in \mathbb{Z}[q]$ and Y is a path from $\emptyset$ to $\emptyset$. So 
$$r=\bigg(\sum_{Y}a_YY\bigg)\upsilon_{\emptyset, L}r' $$
Note that for any $r'$ in the sum,  $\upsilon_{\emptyset, L}r'$ is a path ending in $\emptyset$ and the number of hooks in $\upsilon_{\emptyset, L}r'$ is 
less than that of $r$. By induction, the claim hold for each $\upsilon_{\emptyset, L}r'$  and thus for $r$. This finishes the proof.
\end{proof}

\begin{theorem}\label{A3}
The ideal $\mathcal{K}$ is generated by the sandwich relations, quasi-idempotent relations and the torsion relations of type
{\rm (T1)--(T4)}.  Thus 
we obtain a presentation of $E_q(\fS_4)$ over $\mathbb{Z}[q]$.
\end{theorem}

Before giving a proof of the theorem, we modify Williamson's  algorithm \cite[1.3.4, 2.2.7]{GW1} that computes reduced translation sequences and a basis for $\mathrm{Hom}_{H_q}(x_IH_q, x_JH_q)$ for all $I$ and $J$. See also the discussion at the end of Section 6.
We then use the modified algorithm to produce a set $B_{JI}$ of {\it standard paths} from $I$ to $J$ in the following list. We will show that those standard paths span the algebra $A$ in the proof of Theorem \ref{A3}.
$$\aligned
B_{S, I}&=\{\delta_{S,  I}\}; \; 
B_{\emptyset, \emptyset}=\{\chi_{\underline{w}}\mid w\in W \text{ and } \underline{w} \text{ is a fixed reduced expression of  }  w\};\\
B_{\{1,2\}, \{1, 2\}}&=\{{e}_{\{1, 2\}},\, \upsilon_{\{1, 2\}, S} \,\delta_{S, \{1, 2\}} \};\\
B_{\{1,3\}, \{1, 3\}}&=\{{e}_{\{1, 3\}},  \, \delta_{\{1, 3\}, \emptyset}\,\chi_2\upsilon_{\emptyset, \{1, 3\}}, \, \upsilon_{\{1, 3\}, S} \,\delta_{S, \{1, 3\}} \};\\\,
B_{\{1,2\}, \{1, 3\}}&=\{ \delta_{\{1, 2\}, \{1\}}\upsilon_{\{1\},\{1, 3\}},\, \upsilon_{\{1, 2\}, S} \,\delta_{S, \{1, 3\}} \};\\
B_{\{1,2\}, \{2, 3\}}&=\{ \delta_{\{1, 2\}, \{2\}}\upsilon_{\{2\},\{2, 3\}}, \,\upsilon_{\{1, 2\}, S} \,\delta_{S, \{2, 3\}} \};\\
B_{\{1\}, \{1, 3\}}&=\{ \upsilon_{\{1\}, \{1, 3\}}, \,\delta_{\{1\}, \emptyset} \chi_2\upsilon_{\emptyset,\{1, 3\}}, \,\delta_{\{1\}, \emptyset} \chi_3\chi_2\upsilon_{\emptyset,\{1, 3\}}, \,
\upsilon_{\{1\}, S} \,\delta_{S, \{1, 3\}} \};\\
B_{\{1\}, \{1, 2\}}&=\{ \upsilon_{\{1\}, \{1, 2\}}, \, \upsilon_{\{1\}, \{1, 3\}}\delta_{\{1, 3\}, \{1\}}  \upsilon_{\{1\}, \{1, 2\}}, \,
\upsilon_{\{1\}, S} \,\delta_{S, \{1, 2\}} \};\\
B_{\{1\}, \{2, 3\}}&=\{ \delta_{\{1\}, \emptyset}\upsilon_{\emptyset, \{2,3\}}, \, \upsilon_{\{1\}, \{1, 2\}}\delta_{\{1, 2\}, \{2\}}  \upsilon_{\{2\}, \{2, 3\}}, \,
\upsilon_{\{1\}, S} \,\delta_{S, \{2, 3\}} \};\\
B_{\{2\}, \{1, 2\}}&=\{ \upsilon_{\{2\}, \{1, 2\}}, \,\delta_{\{2\}, \emptyset} \chi_3\upsilon_{\emptyset,\{1, 2\}}, \,\upsilon_{\{2\}, S} \,\delta_{S, \{1, 2\}} \};\\
B_{\{2\}, \{1, 3\}}&=\{ \delta_{\{2\}, \emptyset}\upsilon_{\emptyset, \{1, 3\}}, \, \upsilon_{\{2\}, \{1, 2\}}\delta_{\{1, 2\}, \{1\}}  \upsilon_{\{1\}, \{1, 3\}}, \,
\upsilon_{\{2\}, \{2, 3\}}\delta_{\{2, 3\}, \{3\}}  \upsilon_{\{3\}, \{1, 3\}}, \\ &\,\upsilon_{\{2\}, S} \,\delta_{S, \{1, 3\}} \};\\
B_{\{1\}, \{1\}}&= \{ e_{\{1\}}, \upsilon_{\{1\}, \{1, 2\}}\delta_{\{1, 2\}, \{1\}}, \, \upsilon_{\{1\}, \{1, 3\}}\delta_{\{1, 3\}, \{1\}}, \,   
\upsilon_{\{1\}, \{1, 3\}}\delta_{\{1, 3\}, \emptyset}\chi_2 \upsilon_{\emptyset, \{1\}}, \,\\&
 \delta_{\{1\}, \emptyset}\chi_2\upsilon_{\emptyset, \{1, 3\}}  \delta_{\{1, 3\}, \{1\}}, \,  
\upsilon_{\{1\}, \{1, 3\}}\delta_{\{1, 3\}, \emptyset}\chi_2\upsilon_{\emptyset, \{1, 3\}}  \delta_{\{1, 3\}, \{1\}}, \,
 \upsilon_{\{1\}, S}\,\delta_{S, \{1\}}\};\\
B_{\{1\}, \{2\}}&= \{ \delta_{\{1\}, \emptyset}\upsilon_{\emptyset, \{ 2\}}, \upsilon_{\{1\}, \{1, 3\}}\delta_{\{1, 3\}, \emptyset}\upsilon_{\emptyset, \{ 2\}}, 
\,  \upsilon_{\{1\}, \{1, 2\}}\delta_{\{1, 2\}, \{2\}}, \,
\delta_{\{1\}, \emptyset}\upsilon_{\emptyset, \{2, 3\}}\delta_{\{2, 3\}, \{2\}}, \,\\&
\delta_{\{1\}, \emptyset}\chi_2 \upsilon_{\emptyset, \{1, 3\}}\delta_{\{1, 3\}, \emptyset}\upsilon_{\emptyset, \{ 2\}}, \,
 \upsilon_{\{1\}, \{1, 3\}}\delta_{\{1, 3\}, \{1\}}\upsilon_{\{1\}, \{1, 2\}}\delta_{\{1, 2\}, \{2\}}, \,
 \upsilon_{\{1\}, S}\,\delta_{S, \{2\}}
\};\\ \endaligned$$
$$\aligned
B_{\{1\}, \{3\}}&= \{ \delta_{\{1\}, \emptyset}\upsilon_{\emptyset, \{ 3\}}, \,\delta_{\{1\}, \emptyset}\chi_2\upsilon_{\emptyset, \{ 3\}}, \,
\delta_{\{1\}, \emptyset}\upsilon_{\emptyset, \{2,  3\}}\delta_{\{2, 3\}, \{3\}},  \, \\&
\delta_{\{1\}, \emptyset}\chi_2\upsilon_{\emptyset,  \{1, 3\}}\delta_{\{1, 3\}, \{3\}}, \,
\upsilon_{\{1\}, \{1, 3\}} \delta_{\{1, 3\}, \emptyset} \chi_2\upsilon_{\emptyset,  \{1, 3\}}\delta_{\{1, 3\}, \{3\}}, \,\\&
\upsilon_{\{1\}, \{1, 2\}}\delta_{\{1, 2\}, \{2\}}\upsilon_{\{2\},  \{2, 3\}}\delta_{\{2, 3\}, \{3\}}, \,
 \upsilon_{\{1\}, S}\,\delta_{S, \{3\}}\};\\
B_{\{2\}, \{2\}}&= \{ e_{\{2\}}, \upsilon_{\{2\}, \{1, 2\}}\delta_{\{1, 2\}, \{2\}}, \, \upsilon_{\{2\}, \{ 2, 3\}}\delta_{\{2, 3\}, \{2\}}, \, 
\delta_{\{2\}, \emptyset}\upsilon_{\emptyset, \{1, 3\}}\delta_{\{1, 3\}, \emptyset}\upsilon_{\emptyset, \{ 2\}},  \,   \\&
\delta_{\{2\}, \emptyset}\chi_3\upsilon_{\emptyset,  \{1, 2\}}\delta_{\{1, 2\}, \{2\}}, \,
\delta_{\{2\}, \emptyset}\chi_1\upsilon_{\emptyset,  \{2, 3\}}\delta_{\{2, 3\}, \{2\}}, \, 
 \upsilon_{\{2\}, S}\,\delta_{S, \{2\}}\};\\\\
\endaligned
$$
The complete list for the sets of standard paths between any $I$ and $J$ can be obtained by applying (anti-)automorphisms $\sigma$ and $\tau$ to the sets above. 

\begin{proof}[Proof of Theorem \ref{A3}]
We first claim that any path $\delta_{J, \emptyset} \chi_{\underline w} \upsilon_{\emptyset, I}$ in ${\mathscr P}_{J,I}^\emptyset$ is a  linear combination of the standard paths given above.
Thus, by Lemma  \ref{pathsviathetop}, any path via the empty set  $\emptyset$ is a linear combination of the standard paths. This involves case-by-case checking.
To illustrate, we show that 
$$ \aligned \delta_{\{2\}, \emptyset}\chi_1\chi_3\chi_2\upsilon_{\emptyset, \{1, 3\}}&=\upsilon_{\{2\}, S}\delta_{S, \{1, 3\}}+
q \upsilon_{\{2\}, \{2, 3\}} \,\delta_{\{2, 3\}, \{3\}}  \upsilon_{\{3\}, \{1, 3\}} \\ & + q \upsilon_{\{2\}, \{1, 2\}} \,\delta_{\{1, 2\}, \{1\}}  \upsilon_{\{1\}, \{1, 3\}} +
q(q+1)  \delta_{\{2\}, \emptyset}\upsilon_{\emptyset, \{1, 3\}}. \endaligned $$
Indeed, 

$$
\aligned
{\text{LHS}}&=\delta_{\{2\}, \emptyset}\chi_1  \upsilon_{\emptyset, \{3\}} \cdot \delta_{\{3\}, \emptyset}\chi_2  \upsilon_{\emptyset, \{ 3\}} \cdot  \upsilon_{\{3\}, \{1, 3\}}\\
&\overset{\rm(T1)}=\delta_{\{2\}, \emptyset}\chi_1  \upsilon_{\emptyset, \{ 3\}} (\upsilon_{\{ 3\}, \{2, 3\}}  \delta_{\{2, 3\}, \{3\}} +qe_{\{3\}}) \upsilon_{\{3\}, \{1, 3\}}\\
&=\delta_{\{2\}, \emptyset}\chi_1  \upsilon_{\emptyset, \{ 2\}}\upsilon_{\{ 2\}, \{2, 3\}}  \delta_{\{2, 3\}, \{3\}}  \upsilon_{\{3\}, \{1, 3\}} + 
q\delta_{\{2\}, \emptyset}\chi_1 \upsilon_{\emptyset, \{ 3\}}  \upsilon_{\{3\}, \{1, 3\}}\\
&\overset{\rm(T1)}=( \upsilon_{\{ 2\}, \{1, 2\}}  \delta_{\{1, 2\}, \{2\}} +  qe_{\{2\}} )\upsilon_{\{ 2\}, \{2, 3\}}  \delta_{\{2, 3\}, \{3\}} \upsilon_{\{3\}, \{1, 3\}}
+ q(q+1)\delta_{\{2\}, \emptyset} \upsilon_{\emptyset, \{1, 3\}}\\
&=\upsilon_{\{ 2\}, \{1, 2\}} \cdot \delta_{\{1, 2\}, \{2\}}\upsilon_{\{ 2\}, \{2, 3\}}  \delta_{\{2, 3\}, \{3\}} \upsilon_{\{3\}, \{1, 3\}} + 
q\upsilon_{\{ 2\}, \{2, 3\}}  \delta_{\{2, 3\}, \{3\}} \upsilon_{\{3\}, \{1, 3\}}+
\\&\qquad\quad
+ q(q+1)\delta_{\{2\}, \emptyset} \upsilon_{\emptyset, \{1, 3\}}\vspace{-2ex}\\
&\overset{\rm(T4)}= \upsilon_{\{ 2\}, \{1, 2\}} \upsilon_{\{1, 2\}, S}\delta_{S, \{1, 3\}}+q\upsilon_{\{ 2\}, \{1, 2\}}  \delta_{\{1, 2\}, \{1\}} \upsilon_{\{1\}, \{1, 3\}} +\\&
q\upsilon_{\{ 2\}, \{2, 3\}}  \delta_{\{2, 3\}, \{3\}} \upsilon_{\{3\}, \{1, 3\}} + q(q+1)\delta_{\{2\}, \emptyset} \upsilon_{\emptyset, \{1, 3\}}= \text{RHS}.
\endaligned
$$

Next, by Lemma 
\ref{pathsviathebottom}, any path from $I$ to $J$ via $S$ is a multiple of $\upsilon_{J, S}\delta_{S, I}$, which is a standard path. So it suffices to consider those paths that  go via neither  $S$ nor the empty set $\emptyset$, i.e., 
paths in the following graph
\begin{equation}\label{middleA3}
\xymatrix{
\{1\} \ar@/^/[d] \ar@/_/[dr]  &\{2\}\ar@/^/[dr] \ar@/^/[dl]
&\{3\} \ar@/^/[d]\ar@/^/[dl]  \\
\{1, 2\}  \ar@/^/[ur]  \ar@/^/[u]&
\{1, 3\}  \ar@/^/[ur] \ar@/_/[ul]&\{2, 3\}]\ar@/^/[ul]\ar@/^/[u]\\}
\end{equation}
Further we only need to consider non-standard paths that do not contain 
\begin{itemize}\item[(1)] $\delta_{J, I}\upsilon_{I, J}$, where $I\sqsubset J$ (and $|I|=1$).  This equals $\frac{\pi(J)}{\pi(I)}e_{J}$ by the quasi-idempotent relations. 
So paths containing them can be shortened.
\item[(2)] the hooks $\upsilon_{\{s\}, \{s, t\}}\delta_{\{s, t\}, \{s\}}$ for $\{s, t\}=\{1, 2\}$ or $\{2, 3\}$, 
$\upsilon_{\{3\}, \{1, 3\}}\delta_{\{1, 3\}, \{1\}}$ and its reverse. As these hooks are standard and  paths containing them are a linear combinations of 
paths via the empty set $\emptyset$.  By the claim above, this means that paths containing the hooks are linear combinations of standard paths. 
\end{itemize}
In other words, we just consider non-standard paths in the above graph that does not contain cycles of length 2 and paths $\{s\}\to\{1,3\}\to\{t\}$ ($\{s,t\}=\{1,3\}$).
Hence, any such non-standard path must contain  $$r=\delta_{\{1, 2\}, \{1\}} \upsilon_{\{1\}, \{1, 3\}} \delta_{\{1, 3\}, \{1\}} \upsilon_{\{1\}, \{1, 2\}}\; \text{ or } 
\;t=\delta_{\{1, 2\}, \{2\}}\upsilon_{\{2\}, \{2, 3\}}\delta_{\{2, 3\}\{3\}}\upsilon_{\{3\}, \{1, 3\}},$$ or one of the paths obtained from $r$ and $t$ by symmetry. 
Now, by relations (T3)--(T4) and the claim above, such a path is a linear combination of standard paths, as required. This finishes the proof.
\end{proof}

\end{document}